\documentclass{article}
\usepackage[T1]{fontenc}
\usepackage{xspace,microtype}

\usepackage{amsmath,amsthm,amssymb,enumerate,bbm,color}
\usepackage[pdfpagemode=None,colorlinks=true,linkcolor=blue,citecolor=blue]{hyperref}

\newcommand{\mathd}{\mathrm{d}}
\newcommand{\ff}{\mathcal{B}}
\newtheorem{lemma}{Lemma}
\newtheorem{proposition}{Proposition}
\newtheorem{theorem}{Theorem}

\theoremstyle{remark}
\newtheorem{note}{Note}

\renewcommand{\Re}{\operatorname{Re}}
\renewcommand{\Im}{\operatorname{Im}}
\newcommand{\cl}[1]{\underset{#1}\cdot}
\newcommand{\Nil}{\textup{Nil}_3}

\begin{document}
\bibliographystyle{alpha}

\title{The spinor representation formula in 3 and 4 dimensions}
\author{Pascal Romon\footnote{
Universit\'e Paris-Est, LAMA (UMR 8050), 
UPEMLV, UPEC, CNRS, 
F-77454, Marne-la-Vall\'ee, France.
\hfill \break
pascal.romon@univ-mlv.fr, julien.roth@univ-mlv.fr
} , Julien Roth\footnotemark[\value{footnote}]
}
\date{}
\maketitle

\begin{abstract}
In the literature, two approaches to the Weierstrass representation formula using spinors are known,
one explicit, going back to Kusner~\& Schmitt, and generalized by Konopelchenko and Taimanov, 
and one abstract due to Friedrich, Bayard, Lawn and Roth. In this article, we show that both points of view 
are indeed equivalent, for surfaces in $\mathbb{R}^3$, ${\rm Nil}_3$ and $\mathbb{R}^4$. 
The correspondence between the equations of both approaches is explicitly given and as a consequence
we derive alternate (and simpler) proofs of these previous theorems.
%

\end{abstract}

\medskip

\\



\section{Historical background }\label{sec:intro}

The starting point of all representation formulae of submanifolds is the
celebrated {\emph{Weierstrass--Enneper formula}} for minimal surfaces in
$\mathbb{R}^3$: any minimal immersion $f$ is given locally in conformal
coordinates $z = x + i y$ by
\[ f ( z) = \Re \int^z \left( \frac{1 - g^2}{2}, i \, \frac{1 + g^2}{2},
   g \right) h \mathd z , \]
where $g$ is meromorphic and corresponds to the (stereographic) projection of
the Gauss map, $h$ is holomorphic, and $| h |  ( 1 + | g |^2)$ is neither $0$
nor $+ \infty$. This allows to describe locally (and sometimes globally) all
minimal surfaces, and plays a key role in understanding and classifying them.
However, this formula turns out to characterize many more surfaces than just
the minimal ones: namely, and up to a cosmetic change, all surfaces, provided
they are given in conformal coordinates.

Indeed, in 1979, Kenmotsu \cite{Ke} generalized the formula to any two-dimensional
conformal immersion in $\mathbb{R}^3$, with prescribed (nonzero) mean
curvature $H$, by proving that the Gauss map and mean curvature actually
determine the surface (in the minimal case, an extra function $h$ is
required). To that effect, he studied the equation satisfied by the Gauss map.

In 1996, Kusner \& Schmitt \cite{KS}, following an idea of Sullivan, proposed a
similar albeit simpler expression, called the {\emph{spinor representation
formula}}, which gives locally all conformal immersions in $\mathbb{R}^3$:
\begin{equation}   \tag{W3}
  f (z) = \Re \int^z \left( i \, \frac{\bar{\psi}_2^2 + \psi_1^2}{2},
  \frac{\bar{\psi}_2^2 - \psi_1^2}{2} , \psi_1  \bar{\psi}_2 \right)
  \mathd z \label{eq:Weierstrass-spin}
\end{equation}
where $\psi = ( \psi_1, \psi_2)$ is a $\mathbb{C}^2$-valued map satisfying
the following (nonlinear) ``Dirac equation with potential''
\begin{equation}   \tag{DP}
  \left( \left(\begin{array}{cc}
    0 & \partial / \partial z\\
    - \partial / \partial \bar{z} & 0
  \end{array}\right) + \left(\begin{array}{cc}
    U & 0\\
    0 & V
  \end{array}\right) \right)  \left(\begin{array}{c}
    \psi_1\\
    \psi_2
  \end{array}\right) = 0 \label{eq:Dirac-3D}
\end{equation}
where $U = V = He^{\rho} / 2$ is called the potential, and the metric is
$\mathd s^2 = e^{\rho}  | \mathd z |^2 = ( | \psi_1 |^2 + | \psi_2 |^2)  |
\mathd z |^2$. In the minimal case, $H = 0$, so $\psi_1, \bar{\psi}_2$ are
holomorphic and we identify easily the meromorphic data: $h = i \psi_1^{- 1}$, $g = - i
\bar{\psi}_2 / \psi_1$. Note that $\psi$ is defined up to sign and $\psi_1,
\bar{\psi}_2$ behave like $\sqrt{\mathd z}$ w.r.t. coordinate
change{\footnote{A property which indicates that $\psi$ is {\emph{not}} really
a spinor.}}.
\\

A few years later, Konopelchenko \cite{Ko} and Taimanov \cite{T1} extended that
formulation to conformally parametrized surfaces in $\mathbb{R}^4 \simeq
\mathbb{C}^2$ by
\begin{equation}   \tag{W4}
  f ( z) = \int^z \left(\begin{array}{c}
    s_1 t_1\\
    s_1 t_2
  \end{array}\right) \mathd z + \left(\begin{array}{c}
    - \bar{s}_2  \bar{t}_2\\
    \bar{s}_2  \bar{t}_1
  \end{array}\right) \mathd \bar{z} \label{eq:weierstrass-KT}
\end{equation}
where $s_1, s_2, t_1, t_2$ are complex valued functions, obviously defined up
to a global sign change, satisfying the following Dirac type equation, with
complex potential~$h$
\begin{equation}
  \left\{ \begin{array}{l}
    \partial s_1/\partial \bar{z} = - \bar{h}  \overline{s_2}\\
    \partial \bar{s}_2/\partial z = hs_1
  \end{array} \right. \hspace{1em} \text{and$\hspace{1em} \left\{
  \begin{array}{l}
	 \partial t_1/\partial \bar{z} = - h \overline{t_2}\\
    \partial \bar{t}_2/\partial z = \bar{h} t_1
  \end{array} \right.$} \hspace{1em} \text{where } | h | = \frac{\| \vec{H} \|
  e^{\rho}}{2}  \label{eq:dirac-4D-KT} 
\end{equation}
($\vec{H}$ stands for the mean curvature vector and $e^{\rho}$ is again the
conformal factor). The phase of $h$ will be understood later. A particular
case of this equation was obtained independently by H\'elein \& Romon in
their study of Lagrangian immersions in euclidean space \cite{HeRo}. There,
$t_1 = \cos \frac{\beta}{2}$ and $t_2 = \sin \frac{\beta}{2}$, where $\beta$
is the lagrangian angle $\arg \det \mathd f$. This property has led to
deep results in the theory of Hamiltonian-stationary lagrangian surfaces (see
\cite{Le, McRo} and also the surprising parametrization of Aiyama \cite{A}). Finally,
Taimanov \cite{T2} expanded the formula \eqref{eq:Weierstrass-spin} to three
dimensional homogeneous manifolds. The formula itself barely changes (see below),
as well as the Dirac equation \eqref{eq:Dirac-3D}, except for the potential.
For example, if the ambient space is $\Nil$ endowed with a
left-invariant metric, then
\[ 
U = V = \frac{He^{\rho}}{2} + \frac{i}{4}  ( | \psi_2 |^2 - | \psi_1^2 |)
\]

However, these representation formulae were purely local and computational in
nature, and it was not clear what the relationship they had with the spin
bundle until Friedrich bridged the gap in \cite{F1}. He showed that any surface $M$
immersed in $\mathbb{R}^3$ inherits a spinor $\varphi$ (given an
arbitrary choice of a parallel spinor), which satisfies
Dirac's equation with a potential, and can be used to reconstruct the
immersion. This property was extended by Morel to $\mathbb{S}^3$ and
$\mathbb{H}^3$ \cite{Mo}, by Roth to homogeneous spaces $E_{\kappa, \tau}$ \cite{R},
and most recently by Bayard, Lawn \& Roth to $\mathbb{R}^4$ \cite{BLR}.

These abstract approaches all follow B\"ar's remark \cite{Ba} that an ambient spinor,
restricted to an immersed submanifold $M$, can be identified with an intrinsic
spinor (or, in some cases, with a section of a twisted spin bundle). If the
ambient spinor field $\tilde{\varphi}$ satisfies some particular PDE, e.g. is
parallel or Killing, then the induced spinor $\varphi$ also satisfies a PDE,
typically a Dirac equation with a potential linked to the extrinsic geometric
of $M$. Conversely (and non trivially), the Dirac equation with potential is
shown to imply Gauss and Codazzi's equations (and Ricci's when needed), thus
guaranteeing the (local) existence of an immersion with prescribed extrinsic
data.
\\

Our goal in this paper is to clarify and explicit the relationship between
these two point of view, which has not been done before. 
In particular, we show how the concrete representation formulae and Dirac equations
are actually specific gauge choices of the general abstract formulae, explaining 
their simpler appearance.
We give new and more direct proofs of theorems \ref{thm:FR} and \ref{thm:BLR} below,
avoiding the cumbersome Gauss, Codazzi and Ricci equations (although they are 
present implicitly), including their formulation in $E_{\kappa, \tau}$ written
in \cite{D}. Note that computations cannot be completely avoided.

We will recall, first in three, then in four dimensions, how one obtains the
concrete formulae cited above, then present briefly the induced spinor point
of view, and finally write down the correspondence between these two.
\\

\emph{Acknowledgements.} The authors wish to thank Marie-Am\'elie Lawn 
for her interesting suggestions.
\\


\section{Conformal immersions in dimension three}

We consider in this section the case of surfaces in $\mathbb{R}^3$ and in
metric Lie groups of dimension 3, endowed with a left-invariant metric, and
focus on the Heisenberg group $\Nil=\Nil(\tau)$ where $\tau$ is the fiber torsion 
(see~\cite{D}). When $\tau = 0$ we recover $\mathbb{R}^3$. This approach can be
extended to other groups and homogeneous spaces without much work, and we
direct the reader to \cite{T2}.

\subsection{Coordinates and the Maurer--Cartan form}

The formula \eqref{eq:Weierstrass-spin} is based on a simple remark: Given an
immersion $f : M \rightarrow \mathbb{R}^3$ with coordinate $z = x + i y$, then
$f$ is conformal if and only if the derivative $\partial f / \partial z =
\frac{1}{2}  ( \partial f / \partial x - i \partial f / \partial y)$ lies in the
quadric
\[ Q = \left\{ Z = ( Z_1, Z_2, Z_3) \in \mathbb{C}^3, \; Z_1^2 + Z_2^2 +
   Z_3^2 = 0 \right\} \]
(note that $\mathbb{P} Q \subset \mathbb{C}\mathbb{P}^2$ is the grassmannian
$\text{Gr}_{2, 3}$). We can parametrize $Q$ by two complex numbers with what
is usually called the Segre map
\[ Q \ni Z = \left( \frac{i}{2}  ( \bar{\psi}_2^2 + \psi_1^2), \frac{1}{2}  (
   \bar{\psi}_2^2 - \psi_1^2), \psi_1  \bar{\psi}_2 \right) \]
and one recovers immediately \eqref{eq:Weierstrass-spin}. If the ambient space
$\tilde{M}$ is a Lie group, let $( E_1, E_2, E_3)$ be the orthonormal (moving)
frame generated by left translation from any choice of orthonormal frame at
the identity $e$. Then $( E_1, E_2, E_3)$ trivializes $T M$, in other words,
the tangent vectors can be mapped to $\mathbb{R}^3$ in this frame, or
equivalently, one identifies vectors mapped by the differential of $f^{- 1}$
to $T_e M$ with vectors in $\mathbb{R}^3$ (or $T_e M^\mathbb{C}$ with $\mathbb{C}^3$):
\[ Z = \left(\begin{array}{c}
     Z_1\\
     Z_2\\
     Z_3
   \end{array}\right) \simeq f^{- 1}  \frac{\partial f}{\partial z} \]
and $Z$ lies again in $Q$ if and only if $f$ is (weakly-)conformal. Hence the same
formula \eqref{eq:Weierstrass-spin} holds, with $f^{- 1} \partial f / \partial
z$ playing the role of the derivative:
\begin{equation}
  f^{- 1}  \frac{\partial f}{\partial z} ( z) = \frac{i}{2}  ( \bar{\psi}_2^2
  + \psi_1^2) E_1 + \frac{1}{2}  ( \bar{\psi}_2^2 - \psi_1^2) E_2 + \psi_1 
  \bar{\psi}_2 E_3 . \label{eq:Weierstrass-spin-group}
\end{equation}
Note that integration is less obvious than in the trivial translational group
$\mathbb{R}^3$.
\\

Conversely, the formula \eqref{eq:Weierstrass-spin-group} gives an immersion
when integrated over a simply connected planar domain, provided the
Maurer--Cartan form $\alpha = f^{- 1} \mathd f$ satisfies the integrability
condition $\mathd \alpha + \frac{1}{2}  [ \alpha \wedge \alpha] = 0$. Let us
write down the equation in two explicit cases: $\mathbb{R}^3$, where the
bracket is identically zero, and $\Nil ( \tau)$ where the only nonzero
bracket is $[ E_1, E_2] = 2 \tau E_3$ (and $E_3$ is the vertical fiber unit
tangent vector). We recover the euclidean $\mathbb{R}^3$ by letting $\tau$ go
to $0$. The integrability equation 
boils down to, in the $( E_1, E_2, E_3)$ frame at identity,
\begin{multline} 
\renewcommand*{\arraystretch}{1.2}
    - \frac{\partial}{\partial y}  \left(\begin{array}{c}
     - \frac{1}{2} \Im ( \overline{\psi_{}}_2^2 + \psi_1^2)\\
     \frac{1}{2} \Re ( \bar{\psi}_2^2 - \psi_1^2)\\
     \Re ( \bar{\psi}_2 \psi_1)
   \end{array}\right) + \frac{\partial}{\partial x}  \left(\begin{array}{c}
     - \frac{1}{2} \Re ( \bar{\psi}_2^2 + \psi_1^2)\\
     - \frac{1}{2} \Im ( \bar{\psi}_2^2 - \psi_1^2)\\
     - \Im ( \bar{\psi}_2 \psi_1)
   \end{array}\right) 
   \\
     + 2 \tau \left(\begin{array}{c}
     0\\
     0\\
     \frac{1}{4}  ( | \psi_2 |^2 + | \psi_1 |^2)  ( | \psi_2 |^2 - | \psi_1
     |^2)
   \end{array}\right) = 0
\end{multline}
that is
\[ 
\renewcommand*{\arraystretch}{1.2}
\left\{ \begin{array}{l}
     \Im ( ( \bar{\psi}_2)_y  \bar{\psi}_2) - \Re ( (
     \bar{\psi}_2)_x  \bar{\psi}_2) + \Im ( ( \psi_1)_y \psi_1) -
     \Re ( ( \psi_1)_x \psi_1) = 0
     \\
     - \Re ( ( \bar{\psi}_2)_y  \bar{\psi}_2) - \Im ( (
     \bar{\psi}_2)_x  \bar{\psi}_2) + \Re ( ( \psi_1)_y \psi_1) +
     \Im ( ( \psi_1)_x \psi_1) = 0
     \\
     - \Im [ ( ( \bar{\psi}_2)_x + i ( \bar{\psi}_2)_y) \psi_1] -
     \Im [ ( ( \psi_1)_x + i ( \psi_1)_y)  \bar{\psi}_2] 
     \\
     \hspace*{12em} + \frac{\tau}{2}  ( | \psi_2 |^2 + | \psi_1 |^2)  ( | \psi_2 |^2 - | \psi_1|^2) = 0 \, .
   \end{array} \right. \]
The first two lines amount to $( \psi_2)_z \psi_2 + ( \psi_1)_{\bar{z}} \psi_1
= 0$, which implies the existence of a function $U^{}$ such that
\[ U = \frac{1}{\psi_2}  \frac{\partial \psi_1}{\partial \bar{z}} = -
   \frac{1}{\psi_1}  \frac{\partial \psi_2}{\partial z} \]
(at least when $\psi_1, \psi_2$ do not vanish, and by construction, they do
not vanish simultaneously). We obtain the equation \eqref{eq:Dirac-3D} with
some complex valued potential $U$. However the third line yields
\begin{eqnarray*}
  0 & = & 2 \Im \left[ \frac{\partial \bar{\psi}_2}{\partial \bar{z}}
  \psi_1 + \frac{\partial \psi_1}{\partial \bar{z}}  \bar{\psi}_2 \right] -
  \frac{\tau}{2}  ( | \psi_2 |^2 + | \psi_1 |^2)  ( | \psi_2 |^2 - | \psi_1
  |^2)\\
  & = & 2 \Im [ - \bar{U} | \psi_1 |^2 + U | \psi_2 |^2] -
  \frac{\tau}{2}  ( | \psi_2 |^2 + | \psi_1 |^2)  ( | \psi_2 |^2 - | \psi_1
  |^2)\\
  & = & 2 \Im [ U ( | \psi_1 |^2 + | \psi_2 |^2)] - \frac{\tau}{2}  ( |
  \psi_2 |^2 + | \psi_1 |^2)  ( | \psi_2 |^2 - | \psi_1 |^2)
\end{eqnarray*}
so that $\Im U = \frac{\tau}{4}  ( | \psi_2 |^2 - | \psi_1 |^2)$ in
$\Nil ( \tau)$, and vanishes in $\mathbb{R}^3$. Other
three-dimensional groups follow the same lines (see \cite{T2} for a full list).
Note that the Dirac equation {\emph{per se}} does not fully determine the
potential $U$, but it can be derived from the representation formula.

\subsection{Induced spinors }\label{sec:induced-spinors}

We assume here a basic acquaintance with spinor and Clifford algebras (see
\cite{LM} or \cite{F2} for an introduction). For any spin manifold $M^m$, we denote by
$\text{Spin}(M)$ the spin bundle and 
$\Sigma (M) = \text{Spin} (M) \times_{\chi_m} \Sigma_n$ the associated spinor bundle,
where $\chi_m$ is a representation
of $\text{Spin} (M)$ on $\Sigma_m$. We recall here briefly the idea in \cite{F1}
and \cite{R} in dimension three, and we will give a more explicit formulation in later
sections.

Let $\tilde{M}$ be a three dimensional spin manifold{\footnote{It suffices
here to assume that $\tilde{M}$ is orientable.}}, $M$ an oriented immersed surface in
$\tilde{M}$. Because spinor frames in $M$ can be identified with spinor frames
of $\tilde{M}$ along $M$, and both $\text{Spin} ( 2)$ and $\text{Spin} ( 3)$
have a canonical representation on $\Sigma_2 = \Sigma_3 =\mathbb{C}^2$, it is
possible to identify an ambient spinor $\tilde{\varphi}_{| M }$
restricted to $M$ with an intrinsic spinor $\varphi$ of $M$ (when endowed with
the induced metric). All identifications are isomorphic (by the universality
property) and differ by the corresponding Clifford multiplication. In \cite{R}, the
identification is such that
\[ \forall X \in T M \subset T \tilde{M}, \quad X \cl{M} \varphi = ( X \cdot \nu \cdot
   \tilde{\varphi})_{| M } \]
where $\nu$ is the unit (oriented) normal along $M$, $\cl{M}$ denotes the
Clifford multiplication in $M$ and $\cdot$ the one in $\tilde{M}$. We will
rather use a more direct identification, which will simplify the calculations
below, and satisfies the simpler relation
\begin{equation}
  \forall X \in T M \in T \tilde{M}, X \cl{M} \varphi = ( X \cdot
  \tilde{\varphi})_{| M } \, . \label{eq:Clifford-induit}
\end{equation}
Hence from now on, we will make no difference between Clifford multiplication
in $M$ and in $\tilde{M}$. We will rederive the computations in \cite{R} with this
new convention.
\\

A spinor field is the equivalence class of a couple $(s,u)$, where $s$ is 
a section of $\text{Spin}(\tilde{M})$ (a moving spinor frame) and $u$ is 
a $\Sigma_3$-valued function, under the equivalence relation induced by $\chi_3$: 
\[ 
( s, u) \sim ( s a^{- 1}, \chi_3 ( a) u)
\]
for all $a : \tilde{M} \rightarrow \text{Spin} (3) =\mathbb{S}^3$. 
When $\tilde{M}$ is a Lie group, we trivialize $\text{Spin} ( \tilde{M})$ by
a lift $\tilde{s}$ of the left-invariant moving frame defined above, and
define $\tilde{\varphi}$ to be {\emph{constant}} if $\tilde{\varphi}$ is (the
equivalence class of) the couple $( \tilde{s}, u)$ 
and $u$ is constant. (This property
is obviously independent of the choice of the reference spinor frame
$\tilde{s}$.) In $\mathbb{R}^3$, such a spinor is \emph{parallel}, however that is
not the case in general Lie groups. Again, taking as example $\Nil(\tau)$,
where $\tilde{\nabla}_X E_3 = \tau X \times E_3$, we have, on a constant
spinor field $\tilde{\varphi}$
\[ \tilde{\nabla}_{E_1} \tilde{\varphi} = \frac{\tau}{2} E_1 \cdot
   \tilde{\varphi} \hspace{1em} \tilde{\nabla}_{E_2}  \tilde{\varphi} =
   \frac{\tau}{2} E_2 \cdot \tilde{\varphi} \hspace{1em} \tilde{\nabla}_{E_3} 
   \tilde{\varphi} = - \frac{\tau}{2} E_3 \cdot \tilde{\varphi} \]
because
\[ \tilde{\nabla}_X  \tilde{\varphi} = ( \tilde{s}, X u) + \frac{1}{2} 
   \sum_{1 \le p < q \le 3} \langle \tilde{\nabla}_X E_p, E_q
   \rangle E_p \cdot E_q \cdot \tilde{\varphi} \]
where $\tilde{\nabla}$ denotes both the connection on $T \tilde{M}$ and on
$\Sigma \tilde{M}$, depending on the context. (Note: The case of
$\mathbb{R}^3$ can be obtained by letting $\tau = 0$, and we recover
parallelism.)

Let $M$ be an oriented immersed surface, with normal vector $\nu$ and let $T$ be the
projection of $E_3$ on $T M$, so that $E_3 = T + \lambda \nu$. Let $S = -
\mathd \nu$ be the shape operator. Then, by the spinorial Gauss relation,
\begin{eqnarray*}
  \nabla_X \varphi & = & \tilde{\nabla}_X  \tilde{\varphi}_{| M } -
  \frac{1}{2}  ( SX) \cdot \nu \cdot \varphi 
  \\
  & = & \frac{\tau}{2}  ( X^1 E_1 + X^2
  E_2 - X^3 E_3) \cdot \varphi - \frac{1}{2}  ( SX) \cdot \omega_2 \cdot
  \varphi\\
  & = & \frac{\tau}{2}  ( X \cdot \varphi - 2 \langle X, T \rangle E_3 \cdot
  \varphi) - \frac{1}{2}  ( SX) \cdot \omega_2 \cdot \varphi\\
  & = & \frac{\tau}{2}  \left( - X \cdot \underset{\omega_3}{\underbrace{\nu
  \cdot e_1 \cdot e_2}} \cdot \varphi \right) - \tau \langle X, T \rangle  ( T
  \cdot \varphi + \lambda \nu \cdot \varphi) - \frac{1}{2}  ( SX) \cdot
  \omega_2 \cdot \varphi\\
  & = & \frac{\tau}{2} X \cdot \varphi - \tau \langle X, T \rangle  ( T \cdot
  \varphi + \lambda \omega_2 \cdot \varphi) - \frac{1}{2}  ( SX) \cdot
  \omega_2 \cdot \varphi
\end{eqnarray*}
where $\omega_2 = e_1 \cdot e_2 \in C \ell_2 (M)$ (for any oriented
orthonormal frame $( e_1, e_2)$ of $M$), and $\omega_3 = E_1 \cdot E_2 \cdot
E_3 = e_1 \cdot e_2 \cdot \nu \in C \ell_3 ( \tilde{M})$ are called the
{\emph{real volume elements}} (we used that $\chi_3 ( \omega_3) = - \text{id}$ to
infer $\nu \cdot \varphi = \omega_2 \cdot \varphi$). Finally,
$\tilde{\varphi}$ is constant if and only if
\begin{equation}  \tag{C}
  \nabla_X \varphi = - \frac{1}{2}  ( SX) \cdot \omega_2 \cdot \varphi +
  \frac{\tau}{2} X \cdot \varphi - \tau \langle X, T \rangle  ( T \cdot
  \varphi + \lambda \omega_2 \cdot \varphi) \label{eq:constant-3D} .
\end{equation}
This implies in particular that $| \varphi |$ is constant. Applying the
Dirac operator to~$\varphi$, 
we end up with the following Dirac equation
\begin{equation}
  D\varphi=H\omega_2\cdot\varphi+\tau\lambda(\lambda\varphi- T\cdot\omega_2\cdot\varphi) \label{eq:Dirac-3D-theorique}
\end{equation}
We now quote the following
\begin{theorem}[Friedrich, Roth] \label{thm:FR}
  
  Given a riemannian oriented, simply connected surface $M$, a field of symmetric
  endomorphisms $S$, with trace $2 H$ (and in $\Nil(\tau)$ a vector
  field $T$ and function $\lambda$, such that $\lambda^2 + \| T \|^2 =
  1$). Then the following data are equivalent:
  \begin{enumerate}
    \item an isometric immersion from $M^2$ into $\mathbb{R}^3$ (resp.
    $\Nil(\tau)$) such that $S$ is the shape operator, and $H$ the mean
    curvature (and in $\Nil(\tau)$, $E_3 = T + \lambda \nu$),
    
    \item a non-trivial spinor field $\varphi$ solution of
    \eqref{eq:constant-3D}, where the symmetric tensor $S$ and vector field
    $T$ satisfy, for all $X$
    \begin{equation}
      \nabla_X T = \lambda ( SX - \tau JX), \hspace{2em} X (\lambda) = -
      \langle SX - \tau JX, T \rangle  \label{eq:Daniel}
    \end{equation}
    ($J$ being the complex structure),
    
    \item a non-trivial spinor field $\varphi$ of constant norm, solution of
    the Dirac equation~\eqref{eq:Dirac-3D-theorique}, with $T$, $\lambda$
    satisfying:    
    \[
    \mathd \lambda=-2Q_{\varphi}(T)-B(T)-\tau J T,
    \]
    where $B$ is the tensor defined in an orthonormal frame $\{e_1,e_2\}$ 
    by the following matrix
    \[
    \left( \begin{array}{cc}
    2\tau T_1T_2 & \tau\left( T_1^2-T_2^2\right) 
    \\ \\
    \tau\left( T_1^2-T_2^2\right) & -2\tau T_1T_2
    \end{array} \right) ,
    \] 
    and $Q_{\varphi}$ is the Energy-Momentum tensor associated with 
    the spinor field $\varphi$, defined by:
    \[
    Q_{\varphi}(X,Y):=\frac{1}{2}{\rm Re}
    \left\langle X\cdot\nabla_Y\varphi+Y\cdot\nabla_X\varphi,\varphi/|\varphi|^2\right\rangle \,.
    \]
  \end{enumerate}
\end{theorem}

The proof is technical and requires proving that the Gauss \& Codazzi
equations
\[ 
\begin{array}{cc}
K = \det A + \tau^2  ( 1 - 4 \lambda^2) \hspace{2em} 
&
\text{(Gauss)} 
\\
\nabla_X A Y - \nabla_Y A Y - A [ X, Y] = - 4 \tau^2 \lambda ( \langle Y, T
   \rangle X - \langle X, T \rangle Y) 
& ( \text{Codazzi}) 
\end{array}
\]
are satisfied (with the extra conditions \eqref{eq:Daniel} of Daniel on $T$ and
$\lambda$ for $\Nil(\tau)$).

\subsection{An explicit construction of induced spinors
}\label{sec:construction}

We will now construct explicitly the bundles involved, using quaternions for
shorter notations. Let us identify $\mathbb{R}^3$ with $\Im
\mathbb{H}= \text{Span} ( i, j, k)$ and the Clifford algebra with
\[ C \ell_3 \simeq \left\{ \left(\begin{array}{cc}
     a & b\\
     b & a
   \end{array}\right), a, b \in \mathbb{H} \right\} , \;
   \text{ where }
   \mathbb{R}^3 \simeq \left\{ \left(\begin{array}{cc}
     0 & b\\
     b & 0
   \end{array}\right), b \in \Im \mathbb{H} \right\} \, .
\]
The group $\text{Spin} (3)$ identifies with
\[
\text{Spin}(3) \simeq \left\{ \left(\begin{array}{cc}
     a & 0\\
     0 & a
   \end{array}\right), a \in \mathbb{S}^3 \subset \mathbb{H} \right\}
\] 
and acts on $\mathbb{R}^3$ via matrix multiplication, which amounts to
$\theta_3 (a) : x \mapsto a x a^{- 1} = ax \bar{a}$, for any $a \in \mathbb{S}^3$ 
and $x \in \Im \mathbb{H}$, yielding a twofold cover of $\text{SO} ( 3)$. 
From the above action on $\mathbb{H}^2$, we also infer an irreducible action $\chi_3$
of $\text{Spin} ( 3)$ on $\Sigma_3 =\mathbb{H}$, given by the left
multiplication by $a$, $L_a : q \mapsto a q$. The representation is the
restriction{\footnote{The full expression for $\chi_3$ is $\chi_3 \left(
\left(\begin{array}{cc}
  a & b\\
  b & a
\end{array}\right) \right) = L_{a + b}$.}} of $\chi_3 : C \ell_3 \rightarrow
\text{End} ( \mathbb{H})$ such that $\chi_3 ( x) = L_x$ for $x \in \Im
\mathbb{H}$ and $q \in \mathbb{H}$ (since imaginary quaternions have the
property that $x^2 = - x \bar{x} = - | x |^2$). Recall that the Clifford
multiplication of a spinor by a vector is simply $x \cdot \varphi = \chi_3 (
x) \varphi$. Finally, we identify $\Sigma_3$ with $\mathbb{C}^2$ by choosing
a compatible complex structure; in what follows we let $J = \mathbb{} R_i$,
the right multiplication by $i$ (so that the complex number $i$ will coincide
with the quaternion $i$).

The tangent space at identity $e \in \tilde{M}$ identifies with
$\mathbb{R}^3$ and the spinor frame at $e$ with $\text{Spin} ( 3) \simeq
\mathbb{S}^3$. The trivialisation of $\text{Spin} ( \mathbb{R}^3)$ as
$\mathbb{R}^3 \times \text{Spin} ( 3)$ holds in greater generality for three
dimensional groups $\tilde{M}$ via the group multiplication, where the tangent
space $T_x \tilde{M}$ is mapped canonically to $T_e \tilde{M}$ by left
multiplication by $x^{- 1}$. Equivalently, a choice of a spinor frame at
identity yields by left multiplication a spinor moving frame $\tilde{s}$ on
$\tilde{M}$, which trivializes $\text{Spin} ( \tilde{M})$. Such a spinor
moving frame will be called {\emph{constant}}, and is unique up to right
multiplication by $\text{Spin} ( 3)$.

A spinor field $\tilde{\varphi}$ being the (equivalence class of the) 
couple $(s, u)$ where $s$ is a moving spinor frame and $u$ a map
from $M$ to $\Sigma_3 =\mathbb{H}$, 
the equivalence relation now reads
\[ 
( s, u) \sim ( s a^{- 1}, \chi_3 ( a) u) = ( s a^{-
   1}, a u) 
\]
for all $a : \tilde{M} \rightarrow \text{Spin} ( 3) =\mathbb{S}^3 \subset \mathbb{H}$. 
Clifford multiplication is well-defined in the following way: since any moving spinor
frame $s$ induces by $\theta_3$ an orthonormal moving frame $( E_1,
E_2, E_3)$, and any tangent vector $X$ possesses coordinates $( X_1, X_2,
X_3)$ in this frame, we can identify the tangent bundle $T \tilde{M}$ with
$\text{Spin} ( \tilde{M}) \times_{\theta_3} \mathbb{R}^3$, with the
equivalence relation
\[ ( s, ( X_1, X_2, X_3)) \sim ( s a^{- 1}, \theta_3 ( a) (
   X_1, X_2, X_3)) \]
or, in quaternionic notations $x = X_1 i + X_2 j + X_3 k$: $( s, x)
\sim ( s a^{- 1}, \theta_3 ( a) x) = ( s a^{- 1}, a x a^{- 1})$.
The Clifford multiplication is then
\[ X \cdot \varphi = ( s, x) \cdot ( s, u) = ( s, x
   \cdot u) \]
where $x \cdot u$ is the multiplication in $C \ell_3$. It is immediate that $X
\cdot \varphi$ is defined independently of the frame $s$.

In the following, we pick a constant reference spinor frame $\tilde{s}$ as
above, using the group multiplication, and $\tilde{s}$ lies over the
orthonormal moving frame $( E_1, E_2, E_3)$. In $\Nil(\tau)$, we choose
$\tilde{s}$ such that $E_3$ lies on the vertical fiber. As a consequence, a
spinor field $\tilde{\varphi}$ identifies with a map $u : \tilde{M}
\rightarrow \mathbb{H}$. We will then say that $\tilde{\varphi}$ is
{\emph{constant}} if $u$ is, and note that the property does not depend on a
particular choice of $\tilde{s}$.
\\

Writing $\mathbb{R}^2$ as $\text{Span} ( i, j) \in \mathbb{H}$, the
two-dimensional Clifford algebra $C \ell_2$ identifies with $\mathbb{H}$, and
$\text{Spin} ( 2) \simeq U ( 1) = \{ e^{k \alpha} = (\cos \alpha) \, 1 + (\sin
\alpha) \, \omega_2 \}$, where $\omega_2 = i \cdot j = k$. The double cover
$\theta_2 : \text{Spin} ( 2) \rightarrow \text{SO} ( 2)$ is 
$e^{k \alpha} \mapsto \left( \begin{smallmatrix}
  \cos 2 \alpha & - \sin 2 \alpha\\
  \sin 2 \alpha & \cos 2 \alpha
\end{smallmatrix} \right)$. As representation on $\Sigma_2 =\mathbb{H}$, we choose
the universal extension of $\chi_2 : \mathbb{R}^2 \rightarrow \text{End} (
\mathbb{H})$, $\chi_2 ( x_1 i + x_2 j) = L_{x_1 i + x_2 j}$. In particular,
if $a = e^{k \alpha} \in \text{Spin} ( 2)$, $\chi_2 ( a)$ is again the left
multiplication $L_a$. Again, we endow $\Sigma_2$ with the complex structure $J
= R_i$. Finally, the {\emph{complex volume element}} $\omega_2^{\mathbb{C}} =
\sqrt{- 1} \, i \cdot j$, which squares to the identity, acts as $\chi_3 ( i
\cdot j) J = L_k R_i$ on $\Sigma_2$ and splits $\Sigma_2$ into two
eigenspaces $\Sigma_2^{\pm} = ( 1 \pm j) \mathbb{C}$ of {\emph{positive}} and
{\emph{negative}} spinors{\footnote{Note for later that this splitting is
actually hermitian orthogonal, so that $| v |^2 = | v^{+ 2} | + | v^-
|^2$.\label{norme-psi}}}. On a spin manifold $M$, this allow to view spinors
as square roots of forms. Indeed, let $s$ be a spinor frame at a point $z$,
corresponding to the oriented orthonormal frame $( e_1, e_2)$, and $\varphi =
( s, v)$ be a spinor, with $v = \frac{1 + j}{\sqrt{2}} v^+ + \frac{1 -
j}{\sqrt{2}} v^-$. Then the following forms are well defined :
\[ T_z M \ni X = X_1 e_1 + X_2 e_2 \mapsto ( v^+)^2  ( X_1 + i X_2) \]
as well as
\[ X \mapsto ( \overline{v^-})^2  ( X_1 + i X_2) \hspace{1em} \text{ and }
   \hspace{1em} X \mapsto ( v^+  \overline{v^-})  ( X_1 + i X_2) . \]
Indeed, the gauge change $s \rightarrow se^{- k \alpha}$ induces $v^+
\rightarrow v^+ e^{- i \alpha}$, $v^- \rightarrow v^- e^{+ i \alpha}$ and $X_1
+ i X_2 \rightarrow e^{2 i \alpha}  ( X_1 + i X_2)$ (because $e^{k \alpha}  ( 1
\pm j) = ( 1 \pm j) e^{- i \alpha}$). Obviously these 1-forms are of type $(
1, 0)$.
\\

Consider now an immersion $f : M \rightarrow \tilde{M}$. From the embedding of
$\mathbb{R}^2$ into $\mathbb{R}^3$ as the plane $\nu^{\bot}$, we deduce an
embedding of $\text{Spin} ( 2)$ into $\text{Spin} ( 3)$ such that for any $s
\in \text{Spin} ( 2) \subset \text{Spin} ( 3)$, $\theta_3 ( s)$ is the
rotation of axis $\nu$, whose restriction to the oriented plane $\nu^{\bot}$
is $\theta_2 ( s)$. This embedding carries over to the spinor frame bundles:
In the same way as a tangent frame $( e_1, e_2)$ to $M$ identifies with the
tangent frame $( e_1, e_2, \nu)$ of $\tilde{M}$, where $\nu$ is the unit
normal vector of $M$, a spinor frame $s \in \text{Spin} (M)$ identifies with
some spinor frame in $\text{Spin} ( \tilde{M})$, which, given a reference frame 
$\tilde{s}$, can be written as $\tilde{s} a^{- 1}$, for some $a \in \text{Spin}(3)$. 
Suppose furthermore that $M$ is parametrized by $f$ in conformal
coordinates and that $s ( z)$ lies over the frame $( e^{- \rho} \partial f /
\partial x, e^{- \rho} \partial f / \partial y)$. Moving all vectors to $e$ by
left multiplication by $f ( z)^{- 1}$, the identification $s ( z) \simeq
\tilde{s} ( f ( z)) a ( z)^{- 1}$ reads as{\footnote{In $\mathbb{R}^3$, the
group operation is just the vector sum, and one usually identifies (abusively)
$T\mathbb{R}^3$ with $T_0 \mathbb{R}^3 =\mathbb{R}^3$, so that $f^{- 1}
\mathd f$ is identified with $\mathd f$.}}
\begin{equation}
  \epsilon_1 = e^{- \rho} f^{- 1} \frac{\partial f}{\partial x} = a^{- 1} i a,
  \:\; \epsilon_2 = e^{- \rho} f^{- 1} \frac{\partial f}{\partial y} = a^{- 1}
  j a, \;\; \epsilon_3 = f^{- 1} \nu = a^{- 1} k a
  \label{eq:moving-frame}
\end{equation}
where $( i, j, k)$ stand for the frame at $e$ corresponding to $\tilde{s} (
e)$ (in quaternionic notations). Clearly $a$ is determined (up to sign) by
this property.
\\

The identification of $\Sigma \tilde{M}_{| M }$ with $\Sigma M$
used in {\textsection}\ref{sec:induced-spinors} goes as follows: Since, along~$M$, 
any ambient spinor $\tilde{\varphi} = ( \tilde{s}, u)$ can be written as
the couple $( \tilde{s} a^{- 1}, a u)$, and $\tilde{s} a^{- 1}$ is identified
with the intrinsic spinor frame $s$, $\tilde{\varphi}$ identifies with the
intrinsic spinor $\varphi = ( s, v)$ where $v = a u$. Of all possible choices,
this identification is the one coherent with our choice of Clifford multiplication
\eqref{eq:Clifford-induit}.

Summing up the above, we see that from a ambient spinor field
$\tilde{\varphi}$, one can construct $( 1, 0)$-forms along an immersed surface
$M$. We will now replicate the construction of the previous section in order
to obtain the representation formula. The principle is the following: Starting
with a constant ambient spinor $\tilde{\varphi} = ( \tilde{s}, u)$, the
induced spinor $\varphi = ( s, v) = ( \tilde{s} a^{- 1}, a u)$ determines the
spinor frame $a$, since $a = v u^{- 1}$ and $u$ is fixed. We then deduce $f^{-
1} \mathd f$ from $a$ via \eqref{eq:moving-frame}.
Given a unit quaternion $u$, and writing $v = \frac{1 + j}{\sqrt{2}} v^+ +
\frac{1 - j}{\sqrt{2}} v^-$,
\begin{eqnarray*}
  \epsilon_1 & = & \bar{a} i a\\
  & = & \frac{1}{2} u ( \overline{v^+}  ( 1 - j) + \overline{v^-}  ( 1 + j))
  i ( ( 1 + j) v^+ + ( 1 - j) v^-) u^{- 1}\\
  & = & u ( \overline{v^+} k v^+ - \overline{v^-} k v^- + \overline{v^+} iv^- +
  \overline{v^-} iv^+) u^{- 1}\\
  & = & u ( k ( ( v^+)^2 - ( v^-)^2) + i ( \overline{v^+} v^- +
  \overline{v^-} v^+)) u^{- 1}
\end{eqnarray*}
\begin{eqnarray*}
  \epsilon_2 & = & \bar{a} j a\\
  & = & \frac{1}{2} u ( \overline{v^+}  ( 1 - j) + \overline{v^-}  ( 1 + j))
  j ( ( 1 + j) v^+ + ( 1 - j) v^-) u^{- 1}\\
  & = & u ( \overline{v^+} j v^+ + \overline{v^-} j v^- + \overline{v^+} v^- -
  \overline{v^-} v^+) u^{- 1}\\
  & = & u ( j ( ( v^+)^2 + ( v^-)^2) + ( \overline{v^+} v^- - \overline{v^-}
  v^+)) u^{- 1}
\end{eqnarray*}
\begin{eqnarray*}
  \epsilon_3 & = & \bar{a} k a\\
  & = & \frac{1}{2} u ( \overline{v^+}  ( 1 - j) + \overline{v^-}  ( 1 + j))
  k ( ( 1 + j) v^+ + ( 1 - j) v^-) u^{- 1}\\
  & = & u ( - \overline{v^+} iv^+ + \overline{v^- } iv^- + \overline{v^+}
  k v^- + \overline{v^- } k v^+) u^{- 1}\\
  & = & u ( - i ( | v^+ |^2 - | v^- |^2) + 2 k ( v^+ v^-)) u^{- 1}
\end{eqnarray*}
Choose $u = ( 1 + i - j + k) / 2$, which satisfies $u k \bar{u} = - j, \; u j
\bar{u} = - i, \; u i \bar{u} = k$,
\[ \left\{ \begin{array}{lll}
     \epsilon_1 & = & - \Im ( ( v^+)^2 - ( v^-)^2) i - \Re ( (
     v^+)^2 - ( v^-)^2) j + 2 \Re ( \overline{v^-} v^+) k\\
     & = & - \Im ( ( \overline{v^-})^2 + ( v^+)^2) i + \Re ( (
     \overline{v^-})^2 - ( v^+)^2) j + 2 \Re ( \overline{v^-} v^+) k\\
     \epsilon_2 & = & - \Re ( ( v^+)^2 + ( v^-)^2) i + \Im ( (
     v^+)^2 + ( v^-)^2) j - 2 \Im ( \overline{v^-} v^+) k\\
     & = & - \Re ( ( \overline{v^-})^2 + ( v^+)^2) i - \Im ( (
     \overline{v^-})^2 - ( v^+)^2) j - 2 \Im ( \overline{v^-} v^+) k\\
     \epsilon_3 & = & - 2 \Im ( v^+ v^-) i - 2 \Re ( v^+ v^-) j -
     ( | v^+ |^2 - | v^- |^2) k
   \end{array} \right. \]
Defining $\psi_1 = e^{\rho / 2} v^+$ and $\psi_2 = e^{\rho / 2} v^-$ 
(so that $| \psi_1 |^2 + | \psi_2 |^2 = e^{\rho}$, see footnote~\ref{norme-psi}), 
i.e.
\[ v = e^{- \rho / 2}  \left( \frac{1 + j}{\sqrt{2}} \psi_1 + \frac{1 -
   j}{\sqrt{2}} \psi_2 \right) \]
we have recovered exactly the original representation 
formula~\eqref{eq:Weierstrass-spin-group} (and \eqref{eq:Weierstrass-spin} 
in~$\mathbb{R}^3$).

\begin{note}
  Rigorously speaking, we see that $\psi = ( \psi_1, \psi_2)$ is not a
  spinor, because of the metric factor. But $e^{- \rho / 2} \psi$ is. Hence,
  the integrability equation \eqref{eq:Dirac-3D} reads in terms of $v$ as
  \begin{equation}
    \left\{ \begin{array}{l}
      v^+_{\bar{z}} = - \frac{\rho_{\bar{z}}}{2} v^+ + U v^-\\
      v^-_z = - \frac{\rho_z}{2} v^- - U v^+
    \end{array} \right.  \label{eq:DKTbis}
  \end{equation}
\end{note}

\begin{note}
  A different choice of $u$ would yield a different parametrization of the
  complex curve, or more precisely, a rotation of the formula
  \eqref{eq:Weierstrass-spin-group}. Our choice of $u$ is the one
  corresponding to the classical formula. This particular choice will play a
  role in the next section.
\end{note}

We have thus seen concretely how an intrinsic spinor field induced by a
constant ambient spinor is sufficient to reconstruct entirely the immersion (a
fact already noted by the previous authors, although without the link to the
representation formula). Conversely, given $\psi = ( \psi_1, \psi_2)$, then
formula \eqref{eq:Weierstrass-spin-group} defines locally an immersion of $M$
into $\tilde{M}$, provided the integrability condition \eqref{eq:Dirac-3D}
holds. Let $s = \tilde{s} a^{- 1}$ be the moving spinor frame along $M$
associated with the conformal coordinate $z$ ($s$ is unique up to sign), and
set $\varphi = ( s, v)$ with $v = e^{- \rho/2} \psi$. Then along $M$, the
ambient spinor $\tilde{\varphi} = ( \tilde{s}, a^{- 1} v)$ is constant.
Indeed, the translated moving frame $( \epsilon_1, \epsilon_2, \epsilon_3)$ is
$( \bar{a} i a, \bar{a} j a, \bar{a} k a)$ as well as $( u \bar{v} iv \bar{u}, u
\bar{v} j v \bar{u} , u \bar{v} k v \bar{u})$, by definition of the
representation formula. This implies $a = \pm v u^{- 1}$, so $a^{- 1} v = \pm
u$ is constant. 

It is worth noting that we used \eqref{eq:Dirac-3D} as an integrability 
condition and not as a reformulation of $\nabla \tilde{\varphi} = 0$.
Rather, the spinor representation formula \eqref{eq:Weierstrass-spin-group}
{\emph{per se}} implies that the corresponding ambient spinor is constant,
\emph{provided the immersion exists}, namely that \eqref{eq:Dirac-3D} holds. 
In other words \eqref{eq:Dirac-3D} implies \eqref{eq:constant-3D}. 
There remains to show the converse.


\subsection{Dirac: from abstract to explicit }

Let us write the constancy equation \eqref{eq:constant-3D} in terms of
quaternions. In the local coordinate $z$, with associate spinor frame $s$, the
spinor field $\varphi = ( s, v)$ identifies with $v$ and we write
\[ \begin{split} 
   \nabla_X \varphi & =  ( s, Xv) + \frac{1}{2}  \langle \nabla_X \epsilon_1,
   \epsilon_2 \rangle \omega_2 \cdot \varphi 
   = \left( s, Xv + \frac{1}{2} \langle \nabla_X \epsilon_1, \epsilon_2 \rangle 
   \chi_2 ( \omega_2) v \right) 
   \\
   & = \left( s, Xv + \frac{1}{2} \langle \nabla_X \epsilon_1,
   \epsilon_2 \rangle k v \right) \, .
\end{split} \]
From now on, we will abuse notations and write $\nabla_X v$ instead, dropping
the spinor frame part $s$ when there is no ambiguity. Using the connection
coefficients,
\[
\renewcommand*{\arraystretch}{1.2} 
 \begin{array}{rcl}
\nabla_X v & = & e^{- \rho}  ( X_1 v_x + X_2 v_y) + \frac{e^{- \rho}}{2}  (
   \rho_x X_2 - \rho_y X_1) k v
\\
 - \frac{1}{2}  ( SX) \cdot \omega_2 \cdot v & = & \frac{1}{2}  ( ( h_{11} X_1 +
   h_{12} X_2) j - ( h_{21} X_1 + h_{22} X_2) i) v
\\
 X \cdot v  & = &  \chi_2 ( X) v = ( X_1 i + X_2 j) v
\\
  T \cdot v + \lambda \omega_2 \cdot v  & = &  ( T_1 i + T_2 j) v + \lambda k v
\end{array} 
\]
Finally
\begin{eqnarray*}
  X_1 v_x + X_2 v_y & = & - \frac{1}{2}  ( \rho_x X_2 - \rho_y X_1) k v 
  \\
  && + \frac{e^{\rho}}{2}  ( ( h_{11} X_1 + h_{12} X_2) j - ( h_{21} X_1 + h_{22}
  X_2) i) v
  \\
  &  & + \frac{\tau e^{\rho}}{2}  ( X_1 i + X_2 j) v - \tau e^{\rho}  \langle
  X, T \rangle  ( ( T_1 i + T_2 j) v + \lambda k v) \, .
\end{eqnarray*}
Specializing to $\partial / \partial x$ and $\partial / \partial y$
\begin{equation}     \tag{C'}
  \left\{ \begin{array}{l}
    v_x =  \left( \frac{\rho_y}{2} k + e^{\rho}  \frac{h_{11}}{2} j -
    e^{\rho}  \frac{h_{21}}{2} i + \frac{\tau e^{\rho}}{2} i - \tau e^{\rho}
    T_1  ( T_1 i + T_2 j + \lambda k) \right) v\\
    v_y = \left( - \frac{\rho_x}{2} k + e^{\rho}  \frac{h_{12}}{2} j -
    e^{\rho}  \frac{h_{22}}{2} i + \frac{\tau e^{\rho}}{2} j - \tau e^{\rho}
    T_2  ( T_1 i + T_2 j + \lambda k) \right) v
  \end{array} \right. \label{eq:constancy-v-quaternion}
\end{equation}

\begin{note}
Since $v_x v^{-1}$ and $v_y v^{-1}$ both lie in $\Im \mathbb{H}$, $|v|$ is clearly constant.
From now on, any solution $v$ of \eqref{eq:constancy-v-quaternion} will therefore be normalized 
to norm $1$.
\end{note}
\medskip
Let us now write $v = ( v^+, v^-)$ in the $\Sigma_2^+ \oplus \Sigma_2^-$
splitting, so that left quaternionic multiplication (i.e. Clifford
multiplication) acts as a skew-hermitian matrix, e.g.
\[ ( T_1 i + T_2 j + \lambda k) v = i \left(\begin{array}{c}
     ( T_1 - i T_2) v^-\\
     ( T_1 + i T_2) v^+
   \end{array}\right) - \lambda i \left(\begin{array}{c}
     v^+\\
     - v^-
   \end{array}\right) . \]
Then
\begin{equation}
  \left\{ \begin{array}{lll}
    v^+_{\bar{z}} & = - \frac{\rho_{\bar{z}}}{2} v^+ + \frac{e^{\rho}}{2}  ( H
    + i \tau \lambda^2) v^- + \frac{i \tau \lambda e^{\rho}}{2}  ( T_1 + i T_2)
    v^+ & \\
    v^-_z & = - \frac{\rho_z}{2} v^- + \frac{e^{\rho}}{2}  ( - H + i \tau
    \lambda^2) v^+ - \frac{i \tau \lambda e^{\rho}}{2}  ( T_1 - i T_2) v^- & 
  \end{array} \right. \label{eq:constancy-v}
\end{equation}
which differs quite a bit from \eqref{eq:DKTbis}, unless $\tau = 0$
(immersions into $\mathbb{R}^3$). That is not so surprising, since the
abstract equation \eqref{eq:constant-3D} states the existence of an immersion
for which $\varphi$ is induced from {\emph{any}} constant spinor, while
Konopelchenko \& Taimanov's equation \eqref{eq:Dirac-3D} implies the same
result with $\varphi$ induced from the {\emph{specific}} constant spinor field
$( s, u_0)$, namely $u_0 = ( 1 + i - j + k) / 2$. Indeed, because $k u_0 = u_0
i$, we know that the ambient spinor $\tilde{\varphi}$ satisfies
\[ E_3 \cdot \tilde{\varphi} = ( \tilde{s}, k) \cdot ( \tilde{s}, u_0) = (
   \tilde{s}, k u_0) = ( \tilde{s}, u_0 i) \]
so for the induced spinor
\[ ( E_3 \cdot \tilde{\varphi})_{| M } = E_3 \cdot \varphi = ( T +
   \lambda \omega_2) \cdot ( s, v) = ( s, \chi_2 ( T + \lambda \omega_2) v) =
   ( s, vi), \]
in vector notation,
\[ i \left(\begin{array}{c}
     ( T_1 - i T_2) v^-\\
     ( T_1 + i T_2) v^+
   \end{array}\right) - \lambda i \left(\begin{array}{c}
     v^+\\
     - v^-
   \end{array}\right) = i \left(\begin{array}{c}
     v^+\\
     v^-
   \end{array}\right) \]
or the two (equivalent) colinearity relations
\[ \left\{ \begin{array}{l}
     ( T_1 - i T_2) v^- = ( 1 + \lambda) v^+\\
     ( T_1 + i T_2) v^+ = ( 1 - \lambda) v^- .
   \end{array} \right. \]
Substituting in \eqref{eq:constancy-v}, we obtain
\[ \left\{ \begin{array}{ll}
     v^+_{\bar{z}} & = - \frac{\rho_{\bar{z}}}{2} v^+ + \frac{e^{\rho}}{2}  (
     H + i \tau \lambda) v^-\\
     v^-_z & = - \frac{\rho_z}{2} v^- - \frac{e^{\rho}}{2}  ( H + i \tau
     \lambda) v^+
   \end{array} \right. \]
exactly the equation \eqref{eq:DKTbis}. In other words, \eqref{eq:constant-3D}
is more general, while \eqref{eq:Dirac-3D} is simpler because it assumes
implicitly a given choice of spinor coordinates.

We give now another proof of theorem \ref{thm:FR}, showing that equation
\eqref{eq:constancy-v-quaternion} implies the existence of an immersion, for which
$\varphi$ is the restriction of a constant ambient spinor field
$\tilde{\varphi} = ( \tilde{s}, u)$. We do so without going through the full
construction of~\cite{R}.
\\

The key idea is as follows: if $\varphi = ( s, v)$ is the restriction of some
constant spinor field $\tilde{\varphi} = ( \tilde{s}, u)$, it satisfies the
relation
\[ E_3 \cdot \tilde{\varphi} = ( \tilde{s}, k) \cdot ( \tilde{s}, u) = (
   \tilde{s}, k u) = ( \tilde{s}, u q) \]
where $q = u^{- 1} k u$. We need the following lemmas, that prepares for the
construction of the constants $q$ and $u$.


\begin{lemma}
  \label{thm:existence-q}For any solution $v$ of
  \eqref{eq:constancy-v-quaternion}, where $T, \lambda$ follow Daniel's
  equations~\eqref{eq:Daniel}, there exists an imaginary constant $q$ such
  that
  \begin{equation}
    ( T_1 i + T_2 j + \lambda k) v = v q . \label{eq:commutation}
  \end{equation}
\end{lemma}

\begin{proof} The first equation in \eqref{eq:Daniel} is written as:
  \begin{eqnarray*}
    \nabla_X T 
    & = & e^{- \rho}  \big( ( X_1 T_{1 x} + X_2 T_{1 y} + T_2  ( X_1 \rho_y - X_2
    \rho_x)) e_1 
    \\
    && \hspace*{3em} + ( X_1 T_{2 x} + X_2 T_{2 y} + T_1  ( - X_1 \rho_y + X_2
    \rho_x)) e_2 \big)
    \\
    & = & \lambda ( SX - \tau JX)
     =  \lambda \left( \left(\begin{array}{c}
      h_{11} X_1 + h_{12} X_2\\
      h_{21} X_1 + h_{22} X_2
    \end{array}\right) - \tau \left(\begin{array}{c}
      - X_2\\
      X_1
    \end{array}\right) \right)
  \end{eqnarray*}
  specializing
  \[ \left\{ \begin{array}{l}
       T_{1 x} = \lambda e^{\rho} h_{11} - T_2 \rho_y\\
       T_{1 y} = \lambda e^{\rho}  ( h_{12} + \tau) + T_2 \rho_x\\
       T_{2 x} = \lambda e^{\rho}  ( h_{21} - \tau) + T_1 \rho_y\\
       T_{2 y} = \lambda e^{\rho} h_{22} - T_1 \rho_x
     \end{array} \right. \]
 The second equation in \eqref{eq:Daniel} reads
  \[ \left\{ \begin{array}{l}
       \lambda_x = - e^{\rho}  \langle \left(\begin{array}{c}
         h_{11}\\
         h_{21}
       \end{array}\right) - \tau \left(\begin{array}{c}
         0\\
         1
       \end{array}\right), \left(\begin{array}{c}
         T_1\\
         T_2
       \end{array}\right) \rangle = - e^{\rho}  ( T_1 h_{11} + T_2  ( h_{21} -
       \tau))\\
       \lambda_y = - e^{\rho}  \langle \left(\begin{array}{c}
         h_{12}\\
         h_{22}
       \end{array}\right) - \tau \left(\begin{array}{c}
         - 1\\
         0
       \end{array}\right), \left(\begin{array}{c}
         T_1\\
         T_2
       \end{array}\right) \rangle = - e^{\rho}  ( T_1  ( h_{12} + \tau) + T_2
       h_{22})
     \end{array} \right. \]
  Set $E = T_1 i + T_2 j + \lambda k$,
  \[
  \frac{\partial E}{\partial x} 
    = e^{\rho}  ( h_{11}  ( \lambda i - T_1 k) + ( h_{21} - \tau) (
    \lambda j - T_2 k)) + \rho_y  ( T_1 j - T_2 i)
  \]
Recall from \eqref{eq:constancy-v-quaternion} that $v_x = V v$ for some imaginary quaternion
valued function~$V$; then 
\[
\frac{\partial \bar{v} Ev}{\partial x} = \bar{v} \left( \bar{V} E + \frac{\partial E}{\partial x}
+ E V \right) v
\]
It is long but straightforward to show that $\bar{V} E + \frac{\partial E}{\partial x} + E V$
vanishes identically.
  The same computations holds for the $y$ variable. Clearly $q \in \Im
  \mathbb{H}$. \ 
\end{proof}

\begin{note}
  The equation \eqref{eq:constancy-v-quaternion} could accordingly be
  simplified using \eqref{eq:commutation}, eliminating $T$ in the process.
  However, because right quaternionic multiplication does not commute with the
  complex structure, the new equation may not be as simple as~\eqref{eq:DKTbis}, 
  for which $q = i$.
\end{note}
\medskip

The mere existence of $u$ implies the following relation for solutions of
\eqref{eq:constancy-v-quaternion}:

\begin{lemma}
  \label{thm:vanishing}For any solution $v$ of
  \eqref{eq:constancy-v-quaternion} assuming \eqref{eq:Daniel}
  \[ \lambda \bar{v}  ( T_1 i + T_2 j + \lambda k) v = \langle \bar{v} k v, q
     \rangle q . \]
\end{lemma}

\begin{proof}
  \[ \langle \bar{v} k v, q \rangle = \langle v^{- 1} k v, v^{- 1}  ( T_1 i +
     T_2 j + \lambda k) v \rangle = \langle k, T_1 i + T_2 j + \lambda k
     \rangle = \lambda \]
  so
  \[ \langle \bar{v} k v, q \rangle q - \lambda \bar{v}  ( T_1 i + T_2 j +
     \lambda k) v = \lambda q - \lambda | v |^2 q = 0 \, .
     \]
\end{proof}

We conclude the proof of the theorem by giving an immersion via its
Maurer--Cartan form. Using the notations \eqref{eq:moving-frame} of section
\ref{sec:construction}, we need to prove that the $\Im
\mathbb{H}$-valued 1-form
\[ 
\alpha = e^{\rho} u v^{- 1}  ( i \mathd x + j \mathd y) v u^{- 1} 
\]
is the Maurer--Cartan form $f^{- 1} \mathd f$ of some immersion into
$\Nil(\tau)$, where $u$ is a unit quaternion such that $u^{- 1} k u = q$,
obtained through lemma \ref{thm:existence-q} ($u$ is not unique). The
necessary and sufficient condition for local existence is again the
integrability equation:
\begin{eqnarray*}
  0 & = & - \frac{\partial}{\partial y}  ( e^{\rho} u v^{- 1} i v u^{- 1}) +
  \frac{\partial}{\partial x}  ( e^{\rho} u v^{- 1} j v u^{- 1}) + [ e^{\rho}
  u v^{- 1} i v u^{- 1}, e^{\rho} u v^{- 1} j v u^{- 1}]\\
  & = & e^{\rho} u \big( \rho_x  \bar{v} j v + \bar{v}_x j v + \bar{v} j v_x -
  \rho_y  \bar{v} iv - \bar{v}_y iv - \bar{v} iv_y 
  \\
  && \hspace*{12em} + 2 \tau e^{\rho} \langle
  \bar{v} k v, u^{- 1} k u \rangle u^{- 1} k u \big) u^{- 1}.
\end{eqnarray*}
(We have used the relation of the Lie bracket in $\mathfrak{nil}_3(\tau) \simeq \Im
\mathbb{H}$ with the quaternionic multiplication and scalar product: $[ q,
q'] = 2 \tau \langle q q', k \rangle k$.) Factorizing and substituting from
\eqref{eq:constancy-v-quaternion}, and using $T_1^2 + T_2^2 + \lambda^2=1$,
\begin{eqnarray*}
  0 & = & \rho_x  \bar{v} j v + \bar{v} \left( - \frac{\rho_y}{2} k - e^{\rho} 
  \frac{h_{11}}{2} j + e^{\rho}  \frac{h_{21}}{2} i - \frac{\tau e^{\rho}}{2}
  i + \tau e^{\rho} T_1  ( T_1 i + T_2 j + \lambda k) \right) j v 
  \\
  && + \bar{v} j \left( \frac{\rho_y}{2} k + e^{\rho}  \frac{h_{11}}{2} j - e^{\rho} 
  \frac{h_{21}}{2} i + \frac{\tau e^{\rho}}{2} i - \tau e^{\rho} T_1  ( T_1 i
  + T_2 j + \lambda k) \right) v
  \\
  &  & - \rho_y  \bar{v} iv - \bar{v}  \left( \frac{\rho_x}{2} k - e^{\rho} 
  \frac{h_{12}}{2} j + e^{\rho}  \frac{h_{22}}{2} i - \frac{\tau e^{\rho}}{2}
  j + \tau e^{\rho} T_2  ( T_1 i + T_2 j + \lambda k) \right) iv 
  \\
  && - \bar{v} i
  \left( - \frac{\rho_x}{2} k + e^{\rho}  \frac{h_{12}}{2} j - e^{\rho} 
  \frac{h_{22}}{2} i + \frac{\tau e^{\rho}}{2} j - \tau e^{\rho} T_2  ( T_1 i
  + T_2 j + \lambda k) \right) v
  \\
  &  & + 2 \tau e^{\rho} \langle \bar{v} k v, u^{- 1} k u \rangle u^{- 1} k u
  \\
  & = & 2 \tau e^{\rho}  ( - \lambda \bar{v}  ( T_1 i + T_2 j + \lambda k) v
  + \langle \bar{v} k v, u^{- 1} k u \rangle u^{- 1} k u)
\end{eqnarray*}
which vanishes according to lemma \ref{thm:vanishing}. This implies the
local existence of an immersion in $\Nil(\tau)$ (including the special case of
$\mathbb{R}^3$ when $\tau = 0$), for which $\varphi$ is the restriction of an
ambient spinor.
\begin{note}
A ``colinearity'' relation similar to \eqref{eq:commutation} appears naturally 
when using the $spin^c$ formalism in $\Nil$
\[
\xi \cdot \varphi = (T + \lambda \nu) \cdot \varphi = i \varphi
\]
(see \cite{NR} for details). 
\end{note}

\section{Conformally parametrized surfaces in $\mathbb{R}^4$}

Recently, Bayard, Lawn \& Roth \cite{BLR} have characterized isometric immersions
of surfaces in $\mathbb{R}^4$ by means of spinor fields, giving an spinorial
analog of the Gauss--Ricci--Codazzi equations. More precisely, they show the
following theorem:

\begin{theorem}
  [Bayard, Lawn \& Roth]
  \label{thm:BLR}
  Given a Killing spinor $\tilde{\varphi}$ with Killing constant $\lambda$ 
  in $\mathbb{M}^4(4 \lambda^2)$ (the real space form of curvature $4 \lambda^2$) 
   and an immersed
  surface $M^2$, its restriction $\varphi$ to $M$ lies in the twisted spinor
  bundle $\Sigma = \Sigma M \otimes \Sigma E$ ($E$ being the normal bundle)
  and satisfies the Dirac equation
  \begin{equation}
    D \varphi = \vec{H} \cdot \varphi - 2 \lambda \varphi
    \label{eq:dirac-BLR}
  \end{equation}
  for the Dirac operator $D$ along $M$. Conversely, if $( M^2, g)$ is an
  oriented riemannian manifold, with given spin structure and rank 2 oriented
  spin vector bundle $E$ on $M$, then the following properties are equivalent,
  where $\varphi$ is a section of the twisted spinor bundle $\Sigma = \Sigma M
  \otimes \Sigma E$:
  \begin{enumerate}
    \item $\varphi$ is a solution of the Dirac equation \eqref{eq:dirac-BLR},
    such that $\varphi^+, \varphi^-$ do not vanish and
    \[ X | \varphi^+ | = 2 \Re \langle \lambda X \cdot \varphi^-,
       \varphi^+ \rangle \hspace{1em} \text{and} \hspace{1em} X | \varphi^- |
       = 2 \Re \langle \lambda X \cdot \varphi^+, \varphi^- \rangle, \]
    \item $\nabla_X \varphi = \lambda X \cdot \varphi - \frac{1}{2}  \sum_1^2
    e_j \cdot \ff ( X, e_j) \cdot \varphi$ for some symmetric bilinear $\ff$
    section of $E$, $\varphi^+, \varphi^-$ do not vanish,
    
    \item there exists a local immersion of $M$ with mean curvature $\vec{H} =
    \frac{1}{2} \textup{tr} \ff$, second fundamental form $\ff$, normal bundle $E$
    into $\mathbb{M}^4 ( 4 \lambda^2)$.
  \end{enumerate}
\end{theorem}

The difficult part lies in the converse. Moreover they explicit the
reconstruction of $f$ in terms of $\varphi$, a procedure now known by the
generic term of Weierstrass formula. It must be stressed though that this
proof is entirely abstract.
\\

Partially motivating \cite{BLR}, in $\mathbb{R}^4$ only, we have seen in section
\ref{sec:intro} the {\emph{spinor representation formula}}~\eqref{eq:weierstrass-KT} 
of Konopelchenko and Taimanov, with the condition~\eqref{eq:dirac-4D-KT}, 
also named ``Dirac equation''. We will now bridge the
gap between these two approaches, by rewriting \cite{BLR} using concrete
spinor coordinates. We will prove directly (without actually resorting to the
conditions for isometric immersion of Gauss, Codazzi and Ricci) the existence
of an immersion $f$ derived from a spinor $\varphi$, and connect equation
($\ref{eq:dirac-BLR}$) to \eqref{eq:dirac-4D-KT}. Additionally, it
will be clear how one recovers a parallel spinor associated to $f$ from the
immersion itself. As in the previous sections, we will use quaternions, which
play a natural role in $\mathbb{R}^4$, and allow for shorter and more elegant
formulations. In particular, we express equations \eqref{eq:weierstrass-KT}
and \eqref{eq:dirac-4D-KT} as follows (a formalism due to \cite{He}). Let $f$
be a conformal immersion $f : M \rightarrow \mathbb{R}^4 \simeq \mathbb{H}$.
Since $\text{Spin} ( 4) \simeq \mathbb{S}^3 \times \mathbb{S}^3 \subset
\mathbb{H}^2$ acts on $\mathbb{H}$ via
\[ \forall ( p, q) \in \mathbb{S}^3 \times \mathbb{S}^3, \; \forall x \in
   \mathbb{H}, \hspace{1em} \theta_4 ( p, q) : x \mapsto p x q^{- 1} = p x
   \bar{q} \]
and $( \partial f / \partial x, \partial f / \partial y)$ is an orthonormal
basis, up to a multiplicative factor $e^{\rho}$, there exist $A, B : M
\rightarrow \mathbb{H}^{\ast}$ such that
\begin{equation}   \tag{D}
  \mathd f = \bar{A} \mathd z B = \bar{A}  ( \mathd x + i \mathd y) B = \bar{A}
  B \mathd x + \bar{A} i B \mathd y \label{eq:df}
\end{equation}
and the conformal factor is $e^{\rho} = | \bar{A} B |$. Conversely such an
expression can be integrated locally into a function $f$ if and only if
\begin{multline} 
0 = \mathd ( \mathd f) = ( ( \bar{A} i B)_x - ( \bar{A} B)_y) \mathd x
   \wedge \mathd y 
   \\
   = \bar{A}  \left[ \bar{A}^{- 1}  \left( \frac{\partial
   \bar{A}}{\partial x} + \frac{\partial \bar{A}}{\partial y} i \right) i + i
   \left( \frac{\partial B}{\partial x} + i \frac{\partial B}{\partial y}
   \right) B^{- 1} \right] B \mathd x \wedge \mathd y
\end{multline}
i.e.
\begin{equation}
  \bar{A}^{- 1}  ( \partial \bar{A} / \partial \bar{z}) i + i ( \partial
  \bar{z} \backslash \partial B) B^{- 1} = 0 \label{eq:integrability}
\end{equation}
\begin{note}
  We use here the quaternionic notation for complex derivatives, with $\mathd
  z = \mathd x + i \mathd y$ (and $i \in \mathbb{H}$)
  \[ \mathd f = ( \partial f / \partial z) \mathd z + ( \partial f / \partial
     \bar{z}) \mathd \bar{z} = \mathd z ( \partial z \backslash \partial f) +
     \mathd \bar{z}  ( \partial \bar{z} \backslash \partial f) \]
  equivalently
  \[ \partial f / \partial z = \frac{1}{2}  \left( \frac{\partial f}{\partial
     x} - \frac{\partial f}{\partial y} i \right), \; \partial f / \partial z
     = \frac{1}{2}  \left( \frac{\partial f}{\partial x} + \frac{\partial
     f}{\partial y} i \right), \]
  \[ \partial z \backslash \partial f = \frac{1}{2}  \left( \frac{\partial
     f}{\partial x} - i \frac{\partial f}{\partial y} \right), \; \partial
     \bar{z} \backslash \partial f = \frac{1}{2}  \left( \frac{\partial
     f}{\partial x} + i \frac{\partial f}{\partial y} \right) . \]
\end{note}
\medskip

If we set $A = \bar{s}_1 - j s_2$, $B = t_1 + t_2 j$ (the reason for which is
explained below), then $\mathd f = \bar{A} \mathd z B$, yields exactly the
``Weierstrass formula'' \eqref{eq:weierstrass-KT}
\[ \bar{A} \mathd z B = ( s_1 t_1 \mathd z - \bar{s}_2  \bar{t}_2 \mathd
   \bar{z}) + ( s_1 t_2 \mathd z + \bar{s}_2  \bar{t}_1 \mathd \bar{z}) j \]
when writing $\mathbb{R}^4 \simeq \mathbb{C}^2 \simeq \mathbb{C} \oplus
\mathbb{C}j =\mathbb{H}$. Furthermore, the Dirac equation
\eqref{eq:dirac-4D-KT} amounts to
\begin{equation}  \tag{I}
  ( \partial z \backslash \partial A) A^{- 1} = - h j \hspace{1em} \text{and}
  \hspace{1em} ( \partial \bar{z} \backslash \partial B) B^{- 1} = h j
  \label{eq:KT-quaternion}
\end{equation}
for some (nonconstant) complex potential $h$. Whenever the last equation is
satisfied, the integrability condition~\eqref{eq:integrability} holds immediately:
\[ j \bar{h} i + i h j = h ( j i + i j) = 0 . 
\]
The converse is false: \eqref{eq:integrability} does not imply \eqref{eq:KT-quaternion} 
in general; however one can use the gauge freedom 
$(A, B) \mapsto ( e^{i \alpha} A, e^{i \alpha} B)$ (already noted in \cite{T1}). 
Using that freedom, there is always a choice of $(A, B)$, such that
$(\partial z \backslash \partial A) A^{-1}$ lies in $\mathbb{C} j$ (and therefore 
$( \partial z \backslash \partial B) B^{-1}$ as well, because of \eqref{eq:integrability}).
They are actually opposite, and \eqref{eq:KT-quaternion} holds.
This choice is
unique up to a constant: if $( \partial \bar{z} \backslash \partial B) B^{- 1}
\in \text{Span} ( j, k)$, then
\[ ( \partial \bar{z} \backslash \partial ( e^{i \alpha} B))  ( e^{i \alpha}
   B)^{- 1} = e^{i \alpha}  \left( ( \partial \bar{z} \backslash \partial B)
   B^{- 1} + \frac{\partial \alpha}{\partial \bar{z}} \right) \]
lies in $\text{Span} ( j, k)$ if and only if $\frac{\partial \alpha}{\partial \bar{z}} =
0$.

\subsection{Spinors in two and four dimensions}

In two dimensions, we choose as above for the Clifford representation of
$\mathbb{R}^2 =\mathbb{R}e_1 \oplus \mathbb{R}e_2$ on $\Sigma_2
=\mathbb{H}$ the map $\chi_2$: $\chi_2 ( e_1) = L_i$ and $\chi_2 ( e_2) =
L_j$, where $L_i, L_j$ denote the left quaternionic multiplication by $i$ and
$j$ respectively. However, for compatibility purposes with the above
equations, we endow $\mathbb{H}$ with the complex structure $J = R_j$, 
the right multiplication by $j$. This allows to equate $\mathbb{H}$ with
$\mathbb{C}^2$, but complex numbers in this setting are linear combinations
over $\mathbb{R}$ of quaternions $1$ and $j$. To emphasize this (purely
formal) difference and make notations clear, we will write when needed
$\hat{\mathbb{C}} =\mathbb{R} \oplus j\mathbb{R} \simeq \mathbb{C}$, so
that $\mathbb{H} \simeq \hat{\mathbb{C}}^2$ ({\footnote{Because of the
non-commutativity of $\mathbb{H}$, all vector spaces over $\hat{\mathbb{C}}$
are vector spaces on the right.}}).

As usual, $\Sigma_2$ splits into $\Sigma_2^+ \oplus \Sigma_2^-$, eigenspaces
of $\chi_2 ( \omega_2^{\mathbb{C}})$ w.r.t. the eigenvalues $+ 1$ and $- 1$;
$\omega_2^{\mathbb{C}} = J e_1 \cdot e_2$ is the complex volume, and is mapped
by $\chi_2$ to $L_i L_j R_j = L_k R_j$, whose eigenspaces are
\[ \Sigma_2^+ = ( 1 - i)  \hat{\mathbb{C}}, \hspace{1em} \Sigma_2^- = ( 1 +
   i)  \hat{\mathbb{C}} . \]

The euclidean space $\mathbb{R}^4$ is identified with $\mathbb{H}$ (which
amounts to taking $1, i, j, k$ as reference basis), on which $\text{Spin} ( 4)
\simeq \mathbb{S}^3 \times \mathbb{S}^3 \subset \mathbb{H} \times
\mathbb{H}$ acts as double cover of $S O ( 4)$ via the classical
representation $\theta$: $\theta ( p, q) \cdot x = p x q^{- 1} = p x \bar{q}$.
The Clifford algebra is represented on $\Sigma_4 =\mathbb{H}^2$ via $\chi_4$,
which we define by its restriction to $\mathbb{R}^4 =\mathbb{H}$:
\[ \chi_4 ( x) = \left(\begin{array}{cc}
     0 & x\\
     - \bar{x} & 0
   \end{array}\right) \]
One checks directly that $\chi_4$ induces the following representation of
$\text{Spin} ( 4)$: $\chi_4 ( ( p, q)) = \left(\begin{array}{cc}
  p & 0\\
  0 & q
\end{array}\right)$. Note that $\chi_4 ( x)$ is a complex-linear endomorphism,
provided $\mathbb{H}^2$ is also endowed with $J$, the right multiplication by
$j$, as above, so that $\mathbb{H}^2 \simeq \hat{\mathbb{C}}^4$. The
representation decomposes into two irreducible subspaces, which are here
$\Sigma_4^+ =\mathbb{H} \oplus 0$ and $\Sigma_4^- = 0 \oplus \mathbb{H}$. A
spinor $\tilde{\varphi}$ on $\mathbb{R}^4$ is now the equivalence class of
couples $( ( p, q), ( a, b))$ where $( p, q) \in \text{Spin} ( 4)
=\mathbb{S}^3 \times \mathbb{S}^3$ and $( a, b) \in \Sigma_4
=\mathbb{H}^2$, under the following equivalence relation: for all $( p_1,
q_1) \in \mathbb{S}^3 \times \mathbb{S}^3$
\[ ( ( p, q), ( a, b)) \sim ( ( pp_1^{- 1}, q q_1^{- 1}), \chi_4 ( p_1, q_1)  (
   a, b)) = ( ( pp_1^{- 1}, q q_1^{- 1}), ( p_1 a, q_1 b)) \]
and a spinor field is a section of the quotient bundle $\text{Spin} ( 4)
\times_{\chi_4} \Sigma_4$. Finally, because $\mathbb{R}^4$ is globally
trivial, one can always write $\tilde{\varphi}$ in a preferred spinor frame,
e.g. $( 1, 1)$.
\\

The starting point of \cite{BLR} is the following remark: let $M$ be a
two-dimensional submanifold, and $\tilde{\varphi}$ be an ambient 
parallel\footnote{In $\mathbb{R}^4$, a Killing spinor is simply parallel.} 
spinor, i.e. $\tilde{\nabla} \tilde{\varphi} = 0$, where $\tilde{\nabla}$
denotes the covariant derivative of $\mathbb{R}^4$. Define on $M$ the Dirac
operator by $D = e_1 \cdot \nabla_{e_1} + e_2 \cdot \nabla_{e_2}$, where now
$\nabla$ is the covariant derivative on $M$ (here and in the following, $e_1,
e_2$ will denote any orthonormal frame). Then a simple application of the
spinorial Gauss equation
\[ \tilde{\nabla}_X \tilde{\varphi} = \nabla_X  \tilde{\varphi} + \frac{1}{2}
   \sum_{j = 1}^2 e_j \cdot \ff ( X, e_j) \cdot \tilde{\varphi} \]
yields \eqref{eq:dirac-BLR}.

To make sense of the converse result, starting from intrinsic data, one cannot
just consider the restriction $\tilde{\varphi}_{| M }$ as a spinor
on $M$ (unlike in the three-dimensional case), since the spinor spaces
$\Sigma_2$ and $\Sigma_4$ have different dimensions. Instead, one considers
spinors on $M$ and spinors on the normal bundle $E$. Let us briefly recall the
idea of B\"ar (\cite{Ba}, see also \cite{G} for a quick introduction) and consider
$\mathbb{R}^2 \oplus \mathbb{R}^2$. Each factor is equipped with its
Clifford representation as above, which we denote by $\chi_2$ and $\chi_2'$ on
$\Sigma_2$ and $\Sigma_2'$ (they are identical, the prime just makes it clear
which one we are considering). We let $\gamma$ be the map $\gamma = \chi_2
\otimes \chi'_2 ( \omega_2^{\mathbb{C}}) \oplus \text{Id}_{\Sigma_2} \otimes \chi_2'$ from
$\mathbb{R}^4 =\mathbb{R}^2 +\mathbb{R}^2$ on $\Sigma_2 \otimes \Sigma'_2$.
For $u + v \in \mathbb{R}^2 \oplus \mathbb{R}^2$, $\gamma ( u + v)^2 = - ( |
u |^2 + | v |^2) \text{id}$, hence $\gamma$ extends to a representation of the
Clifford algebra $C \ell ( \mathbb{R}^4)$ on $\Sigma_2 \otimes \Sigma_2'$. by
the universal property, $\gamma$ is equivalent to $\chi_4$: there exists a
isomorphism of algebra $\Psi : \Sigma_2 \otimes \Sigma'_2 \rightarrow
\Sigma_4$ making the diagram commutative
\[ \begin{array}{ccc}
     \Sigma_2 \otimes \Sigma'_2 & \overset{\gamma ( x)}{\longrightarrow} &
     \Sigma_2 \otimes \Sigma'_2\\
     \hspace{1em} \downarrow  \Psi &  & \hspace{1em} \downarrow \Psi\\
     \Sigma_4 & \overset{\chi_4 ( x)}{\longrightarrow} & \Sigma_4
   \end{array} \]
We shall make this correspondence explicit below.

We can now define the {\emph{twisted spinor bundle}} $\Sigma$ on a surface
$M$ with a rank two vector bundle $E$, by taking sections in $\Sigma M \otimes
\Sigma E$. More precisely, let $\sigma = ( s, u)$ be a spinor field on $M$,
i.e. a moving spinor frame $s$ (over an orthogonal frame $e_1, e_2$ of $T M$)
and a vector field $u$ on $\Sigma_2$, up to the equivalence relation $( s, u)
\sim ( s a^{- 1}, \chi_2 ( a) u)$; similarly let $\sigma' = ( s', u')$ be the
equivalence class of a spinor frame $s'$ of $E$ (over an orthogonal frame
$e_3, e_4$ of $E$) and $u'$ a vector field on $\Sigma'_2$. Then, when $M
\subset \mathbb{R}^4$, the couple $( s, s')$ identifies canonically with a
spinor frame of $\mathbb{R}^4$ (over the frame $( e_1, \ldots, e_4)$) and an
element $\sigma \otimes \sigma'$ identifies with (the class of) $( ( s, s'),
\Psi ( u \otimes u'))$. This allows to identify ambient spinors
$\tilde{\varphi}$ restricted to $M$ with sections of $\Sigma M \otimes \Sigma
E$, which are intrinsic objects.

The identification $\Sigma_2 \otimes \Sigma_2' \simeq \Sigma_4$ goes as follows,
using the quaternionic notations. From the natural basis of $\Sigma_2 \simeq
\mathbb{H}$ over $\hat{\mathbb{C}}$: $( 1 - i), ( 1 + i)$, we build the
following basis
\[ ( 1 - i) \otimes ( 1 - i), ( 1 + i) \otimes ( 1 + i), ( 1 - i) \otimes ( 1
   + i), ( 1 + i) \otimes ( 1 - i) \]
which actually spans $( \Sigma_2^+ \otimes {\Sigma'}_2^+) \oplus ( \Sigma_2^-
\otimes{ \Sigma'}_2^-) \oplus ( \Sigma_2^+ \otimes {\Sigma'}_2^-) \oplus (
\Sigma_2^- \otimes {\Sigma'}_2^+)$. The corresponding basis in $\Sigma_4 \simeq
\mathbb{H}^2$ spans $\Sigma_4^{+ +} \oplus \Sigma_4^{- -} \oplus \Sigma_4^{+
-} \oplus \Sigma_4^{- +}$, where $( a, b) \in \Sigma_4^{\varepsilon
\varepsilon'}$ if and only if $\chi_4 ( \omega_2^{\mathbb{C}})  ( a, b) = ( \varepsilon
a, \varepsilon' b)$. Then a possible (but not unique) choice of basis is
\[ \left(\begin{array}{c}
     1 + k\\
     0
   \end{array}\right), \left(\begin{array}{c}
     - i - j\\
     0
   \end{array}\right), \left(\begin{array}{c}
     0\\
     i + j
   \end{array}\right), \left(\begin{array}{c}
     0\\
     - 1 - k
   \end{array}\right) \]
and $\Psi$ is the map sending on basis to the other:
\begin{multline}  \label{eq:Psi}
\Psi (  ( 1 - i) \otimes ( 1 - i) a^{+ +} + ( 1 + i) \otimes ( 1
   + i) a^{- -} + ( 1 - i) \otimes ( 1 + i) a^{+ -} + ( 1 + i) \otimes ( 1 -
   i) a^{- +} 
   \\
   = \left(\begin{array}{c}
     ( 1 + k) a^{+ +} - ( i + j) a^{- -}\\
     ( i + j) a^{+ -} - ( 1 + k) a^{- +}
   \end{array}\right) = \left(\begin{array}{c}
     a\\
     b
   \end{array}\right)
\end{multline}

\subsection{The abstract Dirac equation}

Using this correspondence we can write the Dirac equation, {\emph{using only
intrinsic data}} on $M$ and a given rank two bundle $E$. We will assume that
$M$ is parametrized locally by the conformal coordinate $z = x + i y$, inducing a
frame $( e_1, e_2) = e^{- \rho}  ( \partial / \partial_x, \partial / \partial
y)$, where $\rho$ is the conformal factor, and we will also denote
\[ c_x = \langle \nabla^E_{\partial / \partial_x} e_3, e_4 \rangle = e^{\rho} 
   \langle \nabla^E_{e_1} e_3, e_4 \rangle \;, \hspace{1em} c_y = \langle
   \nabla^E_{\partial / \partial_y} e_3, e_4 \rangle = e^{\rho}  \langle
   \nabla^E_{e_1} e_3, e_4 \rangle \]
the normal connection coefficients, and
\[ \text{for } \beta = 1, 2 \hspace{1em} H_{\beta} = \frac{1}{2}  \langle \ff (
   e_1, e_1) + \ff ( e_2, e_2), e_{\beta} \rangle \]
the coefficients which will (eventually) be the mean curvature vector
coordinates, in a chosen orthonormal frame $( e_3, e_4)$ of $E$. Since
$\nabla^M, \nabla^E$ (resp. $\nabla^{\Sigma M}, \nabla^{\Sigma E}$)
are the connections (resp. spinorial connections) on $M$ and $E$ , then $\nabla
= \nabla^{\Sigma M} \otimes \text{Id}_{\Sigma E} + \text{Id}_{\Sigma M}
\otimes \nabla^{\Sigma E}$ is the connection on the twisted spinor bundle
$\Sigma$. For $\sigma \in \Gamma ( \Sigma M)$,
\[ \nabla_X^{\Sigma M} \sigma = L_X \sigma + \frac{1}{2}  \sum_{\alpha <
   \beta} \langle \nabla^M_X e_{\alpha}, e_{\beta} \rangle e_{\alpha} \cdot
   e_{\beta} \cdot \sigma = L_X \sigma + \frac{1}{2}  \langle \nabla^M_X e_1,
   e_2 \rangle e_1 \cdot e_2 \cdot \sigma \]
so that
\begin{eqnarray*}
  \nabla^{\Sigma M}_{e_1} ( 1 - \varepsilon i) & = & \underset{=
  0}{\underbrace{L_{e_1}  ( 1 - \varepsilon i)}} + \frac{1}{2}  \langle
  \nabla^M_{e_1} e_1, e_2 \rangle i j ( 1 - \varepsilon i)\\
  & = & - \frac{e^{- \rho}}{2}  \frac{\partial \rho}{\partial y} \varepsilon
  ( 1 - \varepsilon i)  ( - j)
\end{eqnarray*}
Similarly,
\[
\begin{array}{l}
\nabla^{\Sigma M}_{e_2} ( 1 - \varepsilon i) = 
\frac{e^{- \rho}}{2}  \frac{\partial \rho}{\partial x} \varepsilon (1 - \varepsilon i)  ( - j)
\\
\nabla^{\Sigma E}_{e_1}  ( 1 - \varepsilon' i) = 
\frac{e^{- \rho}}{2} c_1 \varepsilon'  ( 1 - \varepsilon' i)  ( - j)
\\
\nabla^{\Sigma E}_{e_2}  ( 1 - \varepsilon' i) = \frac{e^{- \rho}}{2} c_2
   \varepsilon'  ( 1 - \varepsilon' i)  ( - j)
\end{array}
\]
Let $D = e_1 \cdot \nabla_{e_1} + e_2 \cdot \nabla_{e_2}$ be the Dirac
operator on $\Sigma M \otimes \Sigma E$, then, 
\[
D ( ( 1 - \varepsilon i) \otimes ( 1 - \varepsilon' i))
= \frac{e^{- \rho}}{2}  ( 1 + \varepsilon i) \otimes ( 1 - \varepsilon'
  i)  \left( \varepsilon \varepsilon'  \frac{\partial \rho}{\partial x} + c_2
  + \left( \varepsilon'  \frac{\partial \rho}{\partial y} - \varepsilon c_1
  \right) j \right)
\]
and for any $\hat{\mathbb{C}}$-valued $a$,
\begin{multline}
D ( ( 1 - \varepsilon i) \otimes ( 1 - \varepsilon' i) a) 
= 
\\
( 1 +   \varepsilon i) \otimes ( 1 - \varepsilon' i) e^{- \rho}  \left( \frac{c_2 +
   \varepsilon \varepsilon' \rho_x + ( \varepsilon'
   \rho_y - c_1 \varepsilon) j}{2} a 
   + \varepsilon
   \varepsilon'  \frac{\partial a}{\partial x} + \varepsilon' j \frac{\partial
   a}{\partial x} \right)
\end{multline}
Summing up,
\begin{eqnarray*}
  D \varphi & = & D ( ( 1 - i) \otimes ( 1 - i) a^{+ +} + ( 1 + i) \otimes ( 1
  + i) a^{- -} 
  \\
  && + ( 1 - i) \otimes ( 1 + i) a^{+ -} + ( 1 + i) \otimes ( 1 - i)  a^{- +})
  \\
  & = & ( 1 - i) \otimes ( 1 - i) e^{- \rho}  
  \left( \frac{c_y - \rho_x + \left( \rho_y
  + c_x \right) j}{2} a^{- +} - a^{- +}_x + j
  a^{- +}_y \right)
  \\
  &  & + ( 1 + i) \otimes ( 1 + i) e^{- \rho}  
  \left( \frac{c_y - \rho_x + \left( - \rho_y - c_x \right) j}{2} a^{+ -} - a^{+ -}_x - j
  a^{+ -}_y \right)
  \\
  &  & + ( 1 - i) \otimes ( 1 + i) e^{- \rho}  \left( \frac{c_y
  + \rho_x + \left( - \rho_y + c_x \right) j}{2}  a^{- -} + a^{- -}_x - j
  a^{- -}_y \right)
  \\
  &  & + ( 1 + i) \otimes ( 1 - i) e^{- \rho}  \left( \frac{c_y
  + \rho_x + \left( \rho_y
  - c_x \right) j}{2} a^{+ +} + a^{+ +}_x + j
  a^{+ +}_y \right)
\end{eqnarray*}
Switching to the notation in $\Sigma_4$, through the isomorphism $\Psi$ in 
\eqref{eq:Psi} (which we will be implicit), we now write $\varphi = ( a, b)$, so
\[
D \varphi = D \left(\begin{array}{c}
    a\\
    b
  \end{array}\right) 
  = e^{- \rho}  \left(\begin{array}{c}
    \frac{\partial b}{\partial x} + i \frac{\partial b}{\partial y} + \left(
    \frac{i c_x - c_y}{2} + \frac{\partial \rho}{\partial \bar{z}} \right) b\\
    - \frac{\partial a}{\partial x} + i \frac{\partial a}{\partial y} + \left(
    - \frac{i c_x + c_y}{2} - \frac{\partial \rho}{\partial z} \right) a
  \end{array}\right) \, .
\]
Since
\begin{eqnarray*}
  \vec{H} \cdot \varphi & = & \left(\begin{array}{cc}
    0 & H_3 j + H_4 k\\
    H_3 j + H_4 k & 0
  \end{array}\right)  \left(\begin{array}{c}
    a\\
    b
  \end{array}\right) = \left(\begin{array}{c}
    ( H_3 j + H_4 k) b\\
    ( H_3 j + H_4 k) a
  \end{array}\right)
\end{eqnarray*}
the Dirac equation is written intrinsically as
\begin{equation}
  \left\{ \begin{array}{l}
    \left( \frac{\partial b}{\partial x} + i \frac{\partial b}{\partial y}
    \right) b^{- 1} = - \frac{\partial \rho}{\partial \bar{z}} + \frac{c_y -
    i c_x}{2} + e^{\rho}  ( H_3 j + H_4 k)\\
    \left( \frac{\partial a}{\partial x} - i \frac{\partial a}{\partial y}
    \right) a^{- 1} = - \frac{\partial \rho}{\partial z} - \frac{c_y +
    i c_x}{2} - e^{\rho}  ( H_3 j + H_4 k)
  \end{array} \right. \label{eq:dirac-sys1}
\end{equation}
Or, using the quaternionic notation for complex derivatives, 
\begin{equation}
  \left\{ \begin{array}{l}
    2 ( \partial z \backslash \partial a) a^{- 1} = - \frac{\partial
    \rho}{\partial z} - \frac{c_y + i c_x}{2} - e^{\rho}  ( H_3 j + H_4 k)\\
    2 ( \partial \bar{z} \backslash \partial b) b^{- 1} = - \frac{\partial
    \rho}{\partial \bar{z}} + \frac{c_y - i c_x}{2} + e^{\rho}  ( H_3 j + H_4
    k)
  \end{array} \right. \label{eq:dirac-sys2}
\end{equation}


\subsection{Building the immersion}

To show that \eqref{eq:dirac-sys1} (or \eqref{eq:dirac-sys2}) implies the
local existence of an immersion with the prescribed metric data, let us
define the following $\mathbb{H}$-valued 1-form on $M$
\[ \xi = e^{\omega} a^{- 1}  ( \mathd x + i \mathd y) b \]
for the spinor field $\varphi=(a,b)$. As in \cite{BLR}, we prove the following
\begin{theorem} \label{thm:3}
  If $( a, b)$ satisfy \eqref{eq:dirac-sys1} {\emph{with $a$ and $b$ of
  constant norm}}, then $\xi$ is closed, and the (locally defined) application
  $f : M \rightarrow \mathbb{H}$ such that $\mathd f = \xi$ induces the
  original metric data. Furthermore, we obtain the restriction of a parallel
  ambient spinor $\tilde{\varphi}$ by setting $\tilde{\varphi} = ( ( p, q), (
  a, b))$ for $( p, q)$ a moving spinor frame on $M$ such that $p q^{- 1} =
  e_1$, $p i q^{- 1} = e_2$, $p j q^{- 1} = e_3$ and $p k q^{- 1} = e_4$.
\end{theorem}

We can prove closedness directly, however we will choose a slightly different
route, in order to connect to the formalism of Konopelchenko and Taimanov.

\begin{proposition}
  If $( a, b)$ satisfy \eqref{eq:dirac-sys2} with $a$ and $b$ of constant
  norm, then the $\mathbb{H}$-valued 1-form $\xi = e^{\omega} a^{- 1}  (
  \mathd x + i \mathd y) b$ can be written as $\xi = \bar{A} \, \mathd z B$, with
  $A, B$ satisfying \eqref{eq:KT-quaternion}. In particular $\xi$ is closed,
  and the conclusions of theorem~\ref{thm:3} hold.
\end{proposition}

\begin{proof}
  Note first that the couples $(a, b)$ and $( A, B)$ are different in nature: $|
  a |$ and $| b |$ are constant and equal by hypothesis (we will assume that
  $| a | = | b | = 1$ for simplicity), while $A, B$ clearly are not, unless
  the induced metric is flat. Nevertheless, $(A, B)$ can be derived from $(a, b)$
  by a process similar to a gauge change, up to the conformal factor. Indeed
  write
  \[ A = e^{u + iv + \rho / 2} a, \hspace{1em} B = e^{- u + iv + \rho / 2} \]
  where $u, v$ are real-valued functions to be determined. Then
  \[ \bar{A} \mathd z B = a^{- 1} e^{u - iv + \rho / 2} \mathd z e^{- u + iv +
     \rho / 2} b = e^{\rho} a^{- 1} \mathd z b = \xi . \]
  To prove our claim that $\xi$ is closed, we need only show that $A, B$
  satisfy the equation \eqref{eq:KT-quaternion} for a suitable choice of $u$
  and $v$.
  \begin{eqnarray*}
    ( \partial z \backslash \partial A) A^{- 1} & = & \frac{\partial ( u + iv +
    \rho / 2)}{\partial z} + e^{u + iv + \rho / 2}  ( \partial z \backslash
    \partial a) a^{- 1} e^{u + iv + \rho / 2}\\
    & = & \frac{\partial u}{\partial z} + i \frac{\partial v}{\partial z} -
    \frac{c_y + i c_x}{4} - \frac{e^{\rho + 2 iv}}{2}  ( H_3 j + H_4 k)
  \end{eqnarray*}
  because complex numbers commute, and commute modulo conjugation with $j$ and
  $k$. Similarly,
  \[ ( \partial \bar{z} \backslash \partial B) B^{- 1} = - \frac{\partial
     u}{\partial \bar{z}} + i \frac{\partial v}{\partial \bar{z}} + \frac{c_y
     - i c_x}{4} + \frac{e^{\rho + 2 iv}}{2}  ( H_3 j + H_4 k) . \]
  We want to solve in $\mathbb{C}$
  \[ \renewcommand*{\arraystretch}{2}
  \left\{ \begin{array}{l}    
       \displaystyle \frac{\partial u}{\partial z} + i \frac{\partial v}{\partial z} =
       \frac{c_y + i c_x}{4}
       \\
       \displaystyle \frac{\partial u}{\partial \bar{z}} - i \frac{\partial v}{\partial
       \bar{z}} = \frac{c_y - i c_x}{4}
     \end{array} \right. \]
  but the second equation is the conjugate of the first one. Under smoothness
  assumption (see~\cite[Theorem 4.2.5]{Ho}), there is always a solution of 
  the $\bar{\partial}$ problem
  \[ \frac{\partial ( u - iv)}{\partial \bar{z}} = \frac{c_y - i c_x}{4} . \]
  Given this solution, we remain with
  \[ ( \partial z \backslash \partial A) A^{- 1} = - h j \hspace{1em} \text{and}
     \hspace{1em} ( \partial \bar{z} \backslash \partial B) B^{- 1} = h j \]
  where
  \[ h = \frac{e^{\rho + 2 iv}}{2}  ( H_3 + H_4 i) \]
  depends on the mean curvature vector. We have obtained exactly the Dirac
  equation of Konopelchenko and Taimanov, which ensures integrability, as we
  proved earlier.
  
  The map $f$ is an immersion because, by construction, $A$ and $B$ are not
  allowed to vanish. There remains to show that $f$ induces on $M$ and its
  normal bundle the same metric data we started with. Because of $\mathd f =
  e^{\rho} a^{- 1} \mathd z b$, the metric on $M$ is clearly the one desired.
  For the normal bundle connection and second fundamental form, just pick up
  the constant spinor field $\tilde{\varphi} = ( ( a^{- 1}, b^{- 1}), ( a, b))
  \sim ( ( 1, 1), ( 1, 1))$. We know from the beginning that
  $\tilde{\varphi}_{| M }$ satisfies the Dirac equation
  \eqref{eq:dirac-BLR}, which in intrinsic coordinates yields
  \eqref{eq:dirac-sys1} for the induced metric and spinor coefficients $a, b$.
  So that $a, b$ satisfy two differential systems, and we can identify the
  right hand side term by term, using the fact that $\mathbb{H}$ is a vector
  space (so quaternions split along the basis $( 1, i, j, k)$). E.g. $-
  \frac{\partial \rho}{\partial \bar{z}} + \frac{c_y - i c_x}{2}$ has to
  coincide with $- \frac{\partial \rho}{\partial \bar{z}} + \frac{\langle
  \nabla^{\bot}_{\partial / \partial_y} e_3, e_4 \rangle - i \langle
  \nabla^{\bot}_{\partial / \partial_x} e_3, e_4 \rangle}{2}$ for the induced
  normal covariant derivative, which proves that both connections are equal.
  The same holds for the mean curvature.
\end{proof}

We therefore reach the same conclusion as \cite{BLR}: The Dirac equation
\eqref{eq:dirac-BLR} on $\Sigma M \otimes \Sigma E$ for a spinor field
$\varphi$ such that $| \varphi^+ | = | \varphi^- |$ is equivalent to the
equations of Gauss, Codazzi and Ricci for surfaces in $\mathbb{R}^4$.


%

\begin{thebibliography}{}

\end{thebibliography}


\bigskip

\begin{thebibliography}{XXXX}

\bibitem[A]{A} R. Aiyama, 
Lagrangian surfaces in the complex 2-space,
\emph{Proceedings of the Fifth International Workshop on Differential Geometry (Taegu)}
25--29, 2000.

\bibitem[B\"a]{Ba} C. B\"ar, 
Extrinsic Bounds for Eigenvalues of the Dirac Operator, 
\emph{Annals of Global Analysis and Geometry} 16(6):573--596, 1998.

\bibitem[BLR]{BLR} P. Bayard, M.-A. Lawn \& J. Roth, 
Spinorial Representation of Surfaces into 4-dimensional Space Forms, 
\href{http://arxiv.org/abs/1210.7386v1}{arXiv:1210.7386}, 
to appear in \emph{Annals of Global Analysis and Geometry}.

\bibitem[D]{D} B. Daniel, 
Isometric immersions into 3-dimensional homogeneous manifolds, 
\emph{Commentarii Mathematici Helvetici} 82(1):87--131, 2007.

\bibitem[F1]{F1} T. Friedrich, 
On the spinor representation of surfaces in Euclidean $3$-space,
\emph{Journal of Geometry and Physics} 28(1-2):143--157, 1998.

\bibitem[F2]{F2} T. Friedrich,
Dirac operators in Riemannian Geometry, 
\emph{Graduate studies in mathematics}, Volume 25, American Mathematical Society, 2000.

\bibitem[G]{G} N. Ginoux,
Op\'erateurs de Dirac sur les sous-vari\'et\'es, 
\emph{PhD thesis}, Universit\'e Henri Poincar\'e, Nancy 1, 2002.

\bibitem[H]{He} F. H\'elein,
On Konopelchenko's representation for surfaces in 4 dimensions, 
\href{http://arxiv.org/abs/math/0104101v1}{arXiv:math.DG/0104101}, 2001.

\bibitem[HRo]{HeRo} F. H\'elein \& P. Romon, 
Weierstrass representation of Lagrangian surfaces in four-dimensional space 
using spinors and quaternions, 
\emph{Commentarii Mathematici Helvetici} 75(4):668--680, 2000.

\bibitem[H\"o]{Ho} L. H\"ormander,
An introduction to complex analysis in several variables, third edition,
\emph{North-Holland Mathematical Library}, 7.  North-Holland Publishing Co., Amsterdam, 1990.

\bibitem[Ke]{Ke} K. Kenmotsu, 
Weierstrass formula for surfaces of prescribed mean curvature, 
\emph{Mathematische Annalen} 245(2):89--99, 1979.

\bibitem[Ko]{Ko} B. G. Konopelchenko,
Weierstrass representations for surfaces in
4D spaces and their integrable deformations via DS hierarchy, 
\emph{Annals of Global Analysis and Geometry} 18(1):61--74, 2000.

\bibitem[KS]{KS} R. Kusner \& N. Schmitt, 
The Spinor Representation of Surfaces in Space,
\href{http://arxiv.org/abs/dg-ga/9610005}{arXiv:dg-ga/9610005}, 1996.

\bibitem[LeRo]{Le} K. Leschke \& P. Romon, 
Darboux transforms and spectral curves of Hamiltonian stationary Lagrangian tori, 
\emph{Calculus of Variations and Partial Differential Equations} 38(1-2):45Ð74, 2010.

\bibitem[LM]{LM} H.B. Lawson \& M.L. Michelson,
Spin Geometry, \emph{Princeton University Press}, Princeton, New Jersey, 1989.

\bibitem[McRo]{McRo} I. McIntosh \& P. Romon, 
The spectral data for Hamiltonian stationary Lagrangian tori in $\mathbb{R}^4$, 
\emph{Differential Geometry and its Applications} 29:125--146, 2011.

\bibitem[Mo]{Mo} B. Morel, Surfaces in $\mathbb{S}^3$ and $\mathbb{H}^3$ via spinors, 
\emph{Actes du s\'eminaire de th\'eorie spectrale et g\'eom\'etrie} 23:131--144, 2005.

\bibitem[NR]{NR} R. Nakad \& J. Roth,
Hypersurfaces of ${\rm Spin^c}$ Manifolds and Lawson-Type Correspondence, 
\emph{Annals of Global Analysis and Geometry} {\bf 42} (3), 421-442, 2012.

\bibitem[R]{R} J. Roth,
Spinorial characterizations of surfaces into 3-homogeneous manifolds, 
\emph{Journal of Geometry and Physics} \textbf{60} 1045-1061, 2010.

\bibitem[T1]{T1} I. A. Taimanov, 
Surfaces in the four-space and the Davey-Stewartson equations, 
\emph{Journal of Geometry and Physics} 56(8):1235--1256, 2006.

\bibitem[T2]{T2} I. A. Taimanov, 
Surfaces in three-dimensional Lie groups in terms of spinors, 
\emph{RIMS Kokyuroku} 1605, 133-150, 2008.

\end{thebibliography}
\end{document}